\documentclass[11pt]{article}
\usepackage{mathrsfs}
\usepackage{hyperref}
\usepackage{amsmath,amsthm,amssymb,cite}
\usepackage{diagbox}
\usepackage{epic,eepic}
\usepackage{graphicx}
\usepackage{soul,color}
\usepackage{cases}
\usepackage{verbatim}
\usepackage{enumerate}
\usepackage{multirow}
\usepackage{indentfirst}
\usepackage{algorithm}
\usepackage{algcompatible}
\usepackage{tikz}
\usetikzlibrary{shapes,arrows,chains}
\makeatletter

\newcommand{\Rmnum}[1]{\expandafter\@slowromancap\romannumeral #1@}
\makeatother
\newtheorem{theorem}{Theorem}[section]
\newtheorem{definition}{Definition}[section]

\newtheorem{corollary}{Corollary}[section]
\newtheorem{lemma}{Lemma}[section]

\newcommand{\algrule}[1][.2pt]{\par\vskip.5\baselineskip\hrule height #1\par\vskip.5\baselineskip}
\allowdisplaybreaks[4]

\begin{document}
\title{Randomized Complete Pivoting for Solving \\ Symmetric Indefinite Linear Systems\thanks{Supported by CSC (grant 201606310121).}}
\author{Yuehua Feng\thanks{School of Mathematical Science, Xiamen University, China. E-mail: {\tt fyh1001@hotmail.com}.}
\and Jianwei Xiao\thanks{Department of Mathematics, University of California, Berkeley. E-mail: {\tt jwxiao@berkeley.edu}.}
\and Ming Gu\thanks{Department of Mathematics, University of California, Berkeley. E-mail: {\tt mgu@math.berkeley.edu}.}
}
%
\maketitle

\begin{abstract}

The Bunch-Kaufman algorithm and Aasen's algorithm are two of the most widely used methods for solving symmetric indefinite linear systems, yet they both are known to suffer from occasional numerical instability due to potentially exponential element growth or unbounded entries in the matrix factorization. In this work, we develop a randomized complete pivoting (RCP) algorithm for solving symmetric indefinite linear systems. RCP is comparable to the Bunch-Kaufman algorithm and Aasen's algorithm in computational efficiency, yet enjoys theoretical element growth and bounded entries in the factorization comparable to that of complete-pivoting, up to a theoretical failure probability that exponentially decays with an oversampling parameter. Our finite precision analysis shows that RCP is as numerically stable as Gaussian elimination with complete pivoting, and RCP has been observed to be numerically stable in our extensive numerical experiments.
\end{abstract}

\medskip
{\small

}
\section{Introduction}

A symmetric matrix $A \in \mathbb{R}^{n \times n}$ is said to be indefinite if it has both positive and negative eigenvalues. Its corresponding linear system $Ax = b$ is called a symmetric indefinite linear system. Such symmetric indefinite linear systems appear in various problems coming from acoustics, electromagnetism, and physics of structures \cite{nedelec2001acoustic}. Symmetric indefinite linear systems are also widely solved in linear least squares problems \cite{bjorck1996numerical} and saddle point problems \cite{benzi2005numerical}.

To effectively solve symmetric indefinite linear systems when $A$ is a dense matrix, there are four well-known algorithms: Bunch-Parlett algorithm \cite{bunch1971direct}, Bunch-Kaufman algorithm \cite{bunch1977some}, bounded Bunch-Kaufman algorithm \cite{ashcraft1998accurate}, and Aasen's algorithm \cite{aasen1971reduction}. Duff and Reid's multifrontal algorithm \cite{duff1983multifrontal} and Liu's sparse threshold algorithm \cite{liu1987partial} can be used to solve sparse symmetric indefinite linear systems. All but Aasen's algorithm use diagonal pivoting to factorize $A$ into block $LDL^T$, where $L$ is unit lower triangular and $D$ is symmetric block diagonal with each block of order $1$ or $2$. Aasen's algorithm factorizes $A$ into $LTL^T$, where $T$ is a symmetric tridiagonal matrix.

The block $LDL^T$ factorization is a generalization of the Cholesky factorization, which requires the input matrix $A$ to be positive semi-definite. While the Cholesky factorization is numerically stable with or without diagonal pivoting \cite{golub2012matrix}, block $LDL^T$ factorization with partial pivoting can have numerical instability issues \cite{ashcraft1998accurate,higham1997stability}. There are a number of strategies to choose permutation matrices for numerical stability. Bunch and Parlett \cite{bunch1971direct} proposed a complete pivoting strategy, which requires searching the whole Schur complement at each stage of the block $LDL^T$ factorization and therefore requires up to $O(n^3)$ comparisons. Bunch \cite{bunch1971analysis} proved that Bunch-Parlett algorithm satisfies a backward error bound almost as good as that for Gaussian elimination with complete pivoting (GECP), and its element growth factor is within a factor $3.07(n-1)^{0.446}$ of Wilkinson's element growth factor bound \cite{wilkinson1961error} for GECP. To reduce the number of comparisons, Bunch and Kaufman devised a partial pivoting strategy, which searches at most two columns at each stage and the number of comparisons is reduced from $O(n^3)$ to $O(n^2)$. However, 
the multipliers can't be controlled with Bunch-Kaufman partial pivoting strategy. To overcome this instability problem,  Ashcraft, Grimes, and Lewis \cite{ashcraft1998accurate} proposed a bounded Bunch-Kaufman algorithm, where multipliers can be bounded near one, at a potentially high cost in comparisons. In the worst-case scenario, bounded Bunch-Kaufman is no better than the Bunch-Parlett algorithm, while in the best-case scenario bounded Bunch-Kaufman costs no more than the Bunch-Kaufman algorithm. LAPACK \cite{anderson1999lapack} routines SYSV and SYSVX, and LINPACK \cite{dongarra1979linpack} routines SIFA and SISL are all based on the Bunch-Kaufman algorithm, whereas LAPACK routines SYSV$\_$rook and SYSV$\_$aa are based on bounded Bunch-Kaufman algorithm and Aasen's algorithm respectively.

Randomization techniques are gaining increasing popularity in numerical linear algebra. Application of randomization in numerical linear algebra includes computing condition estimates \cite{arioli2007partial,kenney1998statistical}, solving linear systems using Monte Carlo methods \cite{dimov2008monte} and Random Butterfly Transformation \cite{baboulin2014efficient, baboulin2013accelerating,baboulin2015dense,parker1995random}. Random projection is one important tool among randomization techniques. Random projections \cite{arriaga1999algorithmic, vempala2005random}, a technique of projecting a set of points from a higher-dimensional space to a randomly chosen lower-dimensional space, approximately preserve pairwise distances \cite{johnson1984extensions} with high probability. Random projections have proven themselves very effective tool in numerical linear algebra \cite{gu2015subspace,halko2011finding, mahoney2011randomized,martinsson2011randomized,   melgaard2015gaussian,woolfe2008fast,xiao2016spectrum}. 

In this paper, we employ random projection techniques to develop a novel pivoting strategy called randomized complete pivoting (RCP). Comparing with the Bunch-Kaufman algorithm and Aasen's algorithm, the RCP algorithm requires a relatively small random projection related overhead that diminishes with matrix dimensions. However, we show that the element growth factor upper bound for the RCP algorithm is comparable to those for GECP and the Bunch-Parlett algorithm. Additionally, we performed an error analysis on the RCP algorithm to demonstrate its numerical stability. 

This paper is organized as follows. Section \ref{Sec:Preliminaries and Background} introduces the preliminaries and background. Section \ref{Sec:RCP} presents the RCP algorithm, its implementations and a floating point operations count analysis. Section \ref{Sec:Analysis of RCP} includes analyses of the RCP algorithm, including reliability analysis of randomized column pivots, analysis on element growth and column growth, and finite precision analysis. In Section \ref{Sec:Numerical experiment}, we report results from our extensive numerical experiments of the RCP algorithm, which demonstrate its effectiveness and reliability as claimed. Finally we draw some conclusions in Section \ref{Sec:Conclusions}. 

\section{Preliminaries and Background} \label{Sec:Preliminaries and Background}
In this section, we introduce some basic notation as well as the Bunch-Kaufman algorithm. 
\subsection{Preliminaries}
Let $\| \cdot \|_p$ and $\| \cdot \|_q$ be two vector norms, their corresponding subordinate matrix norm of a given matrix $A  = \left(a_{i,j}\right) \in \mathbb{R}^{m \times n}$ is  \cite{higham2002accuracy,horn2012matrix}
$$
\| A \|_{p, q} \stackrel{def}{=} {\bf sup}_{0 \neq x \in \mathbb{R}^n} \frac{ \| Ax \|_q }{ \| x \|_p } = {\bf max}_{ \| x \|_p = 1} \| Ax \|_q.
$$

The following cases \cite{higham2002accuracy} will be of special use to us:
\begin{align*}
    &\left\| A \right\|_{1, 2}  = {\bf max}_{1 \le i \le n} \left\| A(:,i) \right\|_2, \quad \left\| A \right\|_{1, \infty} = {\bf max}_{i, j} | a_{i, j } |, \\
    &\left\|AB\right\|_{1,2} \le \left\|A\right\|_2 \left\|B\right\|_{1,2}, \quad 
\left\|AB\right\|_{1,2} \le \left\|A\right\|_{1,2} \left\|B\right\|_1,
\end{align*}
where $A \in \mathbb{R}^{m \times k}$ and $B \in \mathbb{R}^{k \times n}$.

As in the original work of Wilkinson on GECP \cite{wilkinson1961error}, Hadamard's inequality \cite{horn2012matrix} will prove useful in our element growth analysis as well. 
\begin{theorem}\label{Th:Hadamard Inequality} Let $A \in \mathbb{R}^{n \times n}$, then 
$$
\left| \det(A) \right| \le \prod_{j = 1}^n \left\| A(:, j) \right\|_2.
$$
\end{theorem}

Lemma \ref{Le:JL lemma} is a much weakened version of the original Johnson-Lindenstrauss Lemma \cite{johnson1984extensions, vempala2005random}. It has proven itself indispensable in the study of randomized numerical linear algebra \cite{melgaard2015gaussian, xiao2016spectrum}.
\begin{lemma}\label{Le:JL lemma} Let $x \in \mathbb{R}^n$ and $ 0 < \epsilon < 1$. Assume that entries in $\Omega \in \mathbb{R}^{p \times n}$ are independent and identically distributed (i.i.d.) random variables drawn from $\mathcal{N} (0,1)$, then 
\begin{equation*}
\mathbb{P} \left( \sqrt{1 - \epsilon}\, \| x \|_2 \le \left\| \frac{1}{\sqrt{p}} \Omega x \right\|_2 \le \sqrt{1 + \epsilon} \, \| x \|_2 \right) \ge 1 - 2 \exp{ \left( - \frac{(\epsilon^2 - \epsilon^3) p}{4} \right) }.
\end{equation*}
\end{lemma}
Lemma \ref{Le:JL lemma} motivates the following definition \cite{melgaard2015gaussian}.
\begin{definition} \label{De:JL condition}
A given vector $x \in \mathbb{R}^n$ satisfies the $\epsilon$-JL condition under random mapping $\Omega \in \mathbb{R}^{p \times n}$ if
$$
\sqrt{1 - \epsilon} \,\| x \|_2 \le \left\| \frac{1}{\sqrt{p}} \Omega x \right\|_2 \le \sqrt{1 + \epsilon} \, \| x \|_2.
$$
\end{definition}
Intuitively, a vector $x \in \mathbb{R}^n$ satisfies the $\epsilon$-JL condition when its length is approximately preserved under a random projection. For any given failure probability $\delta > 0$, Lemma \ref{Le:JL lemma} asserts that vector $x \in \mathbb{R}^n$ satisfies the $\epsilon$-JL condition as long as 
\[ p \ge \frac{4}{\epsilon^2(1 - \epsilon)} \, {\bf log} \frac{2}{\delta}, \]
regardless of $n$. This log-dependence of $p$ on $\delta$ is at the heart of high reliability of randomized algorithms. 

The law of total probability \cite{zwillinger1999crc} is a classical tool to remove the conditions in conditional probability estimates. 
\begin{theorem}\label{Th:Law of Total Probability}
(Law of Total Probability \cite{zwillinger1999crc}). Assume that events $E_1, \cdots, E_m$ are pairwise disjoint, and that their probabilities sum up to one, then for any event $B$, 
$$
\mathbb{P}(B) = \sum_{i = 1}^{m} \mathbb{P} (B | E_i) \,  \mathbb{P} (E_i).
$$
\end{theorem}

\subsection{The Block \texorpdfstring{$LDL^T$}{Lg} Factorization}
Let $A \in \mathbb{R}^{n \times n}$ be a symmetric indefinite matrix, the block $LDL^T$ factorization of $A$ takes the form 
\begin{eqnarray*}
\Pi A \Pi^T & = & LDL^T \\
& = & \begin{pmatrix}
I &  & & \\
L_{21} & I & & \\
\vdots & \vdots  &\ddots & \\
L_{N1} & L_{N2} & \cdots & I 
\end{pmatrix} \begin{pmatrix}
D_1  &  &  &  \\
&  D_2 &  &  \\
&  &  \ddots &  \\
&   &   &   D_N
\end{pmatrix} \begin{pmatrix}
I & L_{21}^T  & \cdots &  L_{N1}^T \\
& I  &  \cdots &  L_{N2}^T  \\
&   &  \ddots  & \vdots \\
&   &   & I
\end{pmatrix} 
,
\end{eqnarray*}
where $\Pi$ is a permutation matrix, $L$ is a unit lower triangular matrix and $D$ is a symmetric block diagonal matrix with each block $D_i~(1 \le i \le N)$ being $1 \times 1$ or $2 \times 2$.

Let $A^{(k)}$ be the Schur complement of order $n-k+1$ in the block $LDL^T$ factorization process. $A^{(1)}$ is just the input matrix $A$. If $A^{(k)} \in \mathbb{R}^{(n-k+1) \times (n-k+1)}$ is nonzero, there exists an integer $s = 1$ or $2$ and a permutation matrix $\Pi_k$ so that
$$
\Pi_k A^{(k)} \Pi_k^T =  \begin{pmatrix}
A_{11}^{(k)} & {A_{21}^{(k)}}^T \\
A_{21}^{(k)} & A_{22}^{(k)}
\end{pmatrix}, 
$$
with a nonsingular $A_{11}^{(k)} = D_k \in \mathbb{R}^{s \times s}$. This allows one to factorize
\begin{eqnarray*}
&&\Pi_k A^{(k)} \Pi_k^T \\
&=& \left(
\begin{array}{cc}
I_s & O \\
A_{21}^{(k)}{A_{11}^{(k)}}^{-1} & I_{n-k+1-s} 
\end{array}
\right) \left(
\begin{array}{cc}
A_{11}^{(k)} & O \\
O & A^{(k + s)} 
\end{array}
\right) \left(
\begin{array}{cc}
I_s & {A_{11}^{(k)}}^{-1}{A_{21}^{(k)}}^T \\
O & I_{n-k+1-s} 
\end{array}
\right),
\end{eqnarray*}
where the Schur complement is $A^{(k + s)} \stackrel{def}{=} A_{22}^{(k)} - A_{21}^{(k)} \left({A_{11}^{(k)}}\right)^{-1} {A_{21}^{(k)}}^T$. $I_s, I_{n-k+1-s}$ are identity matrices of orders $s$ and $n - k + 1 - s$, respectively. Any element of $A_{21}^{(k)}\left({A_{11}^{(k)}}\right)^{-1}$ will be called a multiplier. We repeat this process on the Schur complement $A^{(k + s)}$, and eventually we will obtain the block $LDL^T$ factorization of $A$, where $\Pi$ is the product of all permutation matrices $\Pi_k$. The computational efficiency and numerical stability of this factorization is highly dependent on how $\Pi_k$ is chosen for each $k$. 

The classical measure of numerical stability in an $LDL^T$ factorization is the {\em element growth factor}. 
\begin{definition}\label{De:definition of elment growth factor}
(Element Growth Factor). Let $A \in \mathbb{R}^{n \times n}$,
\[ {\displaystyle 
\rho_{\textrm{elem}}(A) \stackrel{def}{=} \frac{ {\bf max}_k \left\| A^{(k)} \right\|_{1, \infty} }{ \| A \|_{1, \infty } } = \frac{ {\bf max}_{i, j, k} \left| a^{(k)}_{i, j} \right| }{ {\bf max}_{i, j} | a_{i, j } | } \ge 1.} \]
\end{definition}
The $LDL^T$ factorization is typically considered unstable if $\rho_{\textrm{elem}}(A)$ is very large. A concept closely related to $\rho_{\textrm{elem}}(A)$ is the {\em column norm growth factor}. 
\begin{definition}\label{De:definition of column norm growth factor}
(Column Norm Growth Factor). Let $A \in \mathbb{R}^{n \times n}$,
\[ {\displaystyle 
\rho_{\textrm{col}}(A) \stackrel{def}{=} \frac{ {\bf max}_k \left\| A^{(k)} \right\|_{1, 2} }{ \| A \|_{1, 2} } = \frac{ {\bf max}_{j, k} \left\| A^{(k)}(:, j) \right\|_2 }{{\bf max}_{j} \| A(:, j) \|_2}.} \]
\end{definition}
These growth factors differ by at most a factor of $~\sqrt{n}$ according to Lemma \ref{Le:Inequality between element GF and column norm GF}, whose proof we omit. 
\begin{lemma}\label{Le:Inequality between element GF and column norm GF}
Let $A \in \mathbb{R}^{n \times n}$,
$$
\frac{1}{\sqrt{n}} \, \rho_{\textrm{elem}}(A) \le \rho_{col}(A) \le \sqrt{n} \, \rho_{\textrm{elem}}(A).
$$
\end{lemma}

\subsection{The Bunch-Kaufman Algorithm}\label{Sec:BKA}
The Bunch-Kaufman Algorithm is an algorithm for computing the $LDL^T$ factorization with a computationally efficient pivoting strategy, to be called Bunch-Kaufman partial pivoting (BKPP) strategy, depicted in Figure \ref{alg:bkpp}. 
To simplify the notation in Figure \ref{alg:bkpp}, we have suppressed the superscript in entries of $A^{(k)}$ so
$A^{(k)} = (a_{i j})$ for $k \le i,~j \le n$.
Let\footnote{In case multiple indexes of $r$ give the same $\lambda$ value, we choose $r$ to be the smallest index.}
\[ \lambda  = \left|a_{r, k}\right|, \quad \mbox{where} \quad r \stackrel{def}{=} {\bf argmax}_{k +1 \le i \le n} \left|a_{i, k}\right|.\]
Thus $a_{r k}$ is the off-diagonal entry in the first column of $A^{(k)}$ with the largest magnitude. Further define 
\[\sigma = {\bf max}_{k \le i \le n, \; i \not= r} \left|a_{i, r}\right|. \]
It follows that $\sigma \ge \lambda$. The schematic matrix in Figure \ref{Fig:Bunch-Kaufman partial pivot entries} shows the pivot entries and their indexes. 
\begin{figure}[htbp]
\centering
\fbox{\parbox{0.85\linewidth}{
\begin{algorithmic} 
\STATE Choose $0 < \alpha  < 1$;
\IF {$\lambda = 0$}
\STATE \textbf{(0)}: there is nothing to do on this stage of the elimination;
\ELSIF {$| a_{kk} | \ge \alpha \lambda$}
\STATE \textbf{(1)}: use $a_{kk}$ as a $1 \times 1$ pivot;
\ELSIF {$| a_{kk} | \sigma \ge \alpha \lambda^2$}
\STATE \textbf{(2)}: use $a_{kk}$ as a $1 \times 1$ pivot;
\ELSIF {$| a_{rr} | \ge \alpha \sigma$}
\STATE \textbf{(3)}: use $a_{rr}$ as a $1 \times 1$ pivot (swap rows and columns $k$ and $r$ of $A$);
\ELSE
\STATE \textbf{(4)}: use $\left[
\begin{array}{cc}
a_{kk} & a_{rk} \\
a_{rk} & a_{rr}
\end{array}
\right]$ 
as a $2 \times 2$ pivot (swap rows and columns $k + 1$ and $r$ of $A$);
\ENDIF
\end{algorithmic}
}}
\caption{Bunch-Kaufman partial pivoting strategy}\label{alg:bkpp}
\end{figure}
\begin{figure}[htbp]
$$
\left[ 
\begin{array}{cccccc}
a_{kk} & \cdots & a_{rk} & \cdots & \cdots & \cdots \\
\vdots & \ddots & \vdots &  &  &  \\
a_{rk} & \cdots & a_{rr} & \cdots & \pm \sigma & \cdots \\
\vdots & \ddots & \vdots & \ddots & &  \\
\vdots & \cdots & \pm \sigma & & \ddots &  \\
\vdots & \cdots & \vdots & &  & \ddots
\end{array}
\right]
$$
\caption{Bunch-Kaufman partial pivot entries} \label{Fig:Bunch-Kaufman partial pivot entries}
\end{figure}

The BKPP strategy is justified in \cite{bunch1977some}. However, we notice that the BKPP strategy strives to avoid ill-conditioning in the $2\times 2$ pivots, but does not completely succeed in doing so, leading to potential numerical instability beyond large element growth \cite{higham1997stability}. The parameter $\alpha$ is usually chosen as $\alpha = (1 + \sqrt{17}) / 8$, which makes the bound for element growth over two consecutive $1 \times 1$ pivots equal the bound for element growth over one $2 \times 2$ pivot \cite{bunch1977some}.

\section{Randomized Complete Pivoting Algorithm}\label{Sec:RCP}
Instead of finding the larg-est magnitude entry in the whole matrix $A$, as required in GECP, the RCP algorithm chooses the column with the maximum $2$-norm on a random projection of $A$, and then applies a simplified BKPP strategy 
on the chosen column. More precisely, given a symmetric matrix $A \in \mathbb{R}^{n\times n}$ and a random matrix $\Omega \in \mathbb{R}^{p \times n}$ whose entries are sampled from $\mathcal{N}(0,1)$ independently, we compute a random projection of $A$,
\begin{equation*}
B = \Omega \, A \in \mathbb{R}^{p \times n},
\end{equation*}
which has a much smaller row dimension than $A$ for $p \ll n$, the only case that is of interest to us. Instead of finding a pivot on $A$, we choose the column pivot on $B$ by finding its column with the largest column $2$-norm. We denote the permutation of swapping the first column of $B$ and the column of $B$ with the largest $2$-norm by $\widetilde{\Pi}$, then $B \, \widetilde{\Pi}$ satisfies
\begin{equation}\label{Eq:Apply Pi on A B Omega}
B \, \widetilde{\Pi} = \left(\Omega \, \widetilde{\Pi}\right)\left(\widetilde{\Pi}^T\,  A\,  \widetilde{\Pi}\right).
\end{equation}
After applying this permutation $\widetilde{\Pi}$ to the columns and rows of $A$, we perform a symmetric factorization of $\widetilde{\Pi}^T A\, \widetilde{\Pi}$, i.e.,
\begin{equation}\label{Eq:Perform LDL^T on Pi A Pi^T}
\widetilde{\Pi}^T \, A \, \widetilde{\Pi} = \widehat{\Pi}^T L D L^T \widehat{\Pi}.
\end{equation}

Figure \ref{alg:idea} is a graphical illustration of the main idea behind the RCP algorithm. The strategy for choosing $\widehat{\Pi}$ is described in Figure \ref{alg:SBKP}, and we name it simplified Bunch-Kaufman pivoting (SBKP) strategy , where $\widehat{\Pi}$ is called an SBKP permutation. 

Unlike BKPP, the column permutation $\widetilde{\Pi}$ ensures with high probability that $\lambda = O\left(\left\|A\right\|_{1,\infty}\right)$. When SBKP chooses a $2\times 2$ pivot, it is under the conditions that
\[|a_{k,k}| < \alpha \, \lambda, \quad |a_{r,r}| < \alpha \, \lambda, \quad \mbox{with}\quad \lambda = |a_{r,k}|. \]
Thus the $2\times 2$ pivot $\left[
\begin{array}{cc}
a_{kk} & a_{rk} \\
a_{rk} & a_{rr}
\end{array}
\right]$ 
as a $2 \times 2$ must be well-conditioned for $0 < \alpha < 1$. We will discuss the choice $\alpha = \sqrt{2} / 2$ in Subsection \ref{Sec:Analysis of growth factor of RCP}. 
\begin{figure}[htbp]
\centering
\begin{tikzpicture}
\draw (0, 0) rectangle (2, 2);
\node[above] at (1, 2) {$n$};
\node[left] at (0, 1) {$n$};
\node[below] (s2) at (1, 0) {$A$};
\node[below=0.25cm] at (s2) {\begin{tabular}{c} $LDL^T$ \\ factorization \\ \end{tabular}};
\draw (6, 0.75) rectangle (8, 1.25);
\node[right] at (8, 1) {$p$};
\node[above] at (7, 1.25) {$n$};
\node[below] (s1) at (7, 0.75) {$B = \Omega \, A$};
\draw[->, line width=1pt] (2, 1.15) -- node[above] {Random projection} (6, 1.15);
\draw[->, line width=1pt] (6, 0.85) --node [below] {\begin{tabular}{l} One pivot \\ \end{tabular}}  (2, 0.85);
\node[below=0.15cm] (s3) at (s1) {\begin{tabular}{c} maximum \\ column $2$-norm \\ \end{tabular}};
\end{tikzpicture}
\caption{the RCP algorithm takes a column pivot for $B$ as a column pivot for $A$ }\label{alg:idea}
\end{figure}
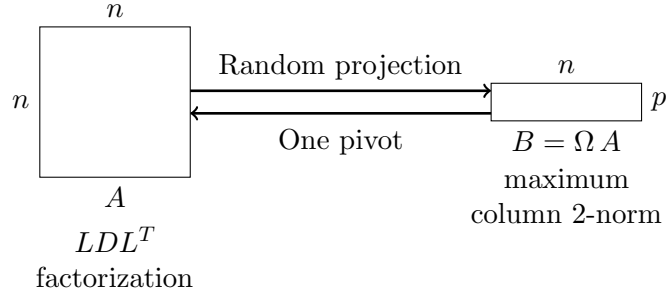 
\begin{figure}[htbp]
\centering
\fbox{\parbox{0.9\linewidth}{
\begin{algorithmic} 
\STATE $ \alpha = \frac{\sqrt{2}}{2} $;
\STATE $\lambda \stackrel{def}{=} |a_{rk}|, \quad \mbox{where} \quad r = {\bf argmax}_{k + 1 \le j \le n} \left| A^{(k)}(j, k) \right|$;
\IF {$\lambda = 0$}
\STATE \textbf{(0)}: there is nothing to do on this stage of the elimination ; 
\ELSIF {$| a_{kk} | \ge \alpha \lambda$}
\STATE \textbf{(1)}: use $a_{kk}$ as a $1 \times 1$ pivot ($s = 1$);
\ELSIF {$| a_{rr} | \ge \alpha \lambda$}
\STATE \textbf{(2)}: use $a_{rr}$ as a $1 \times 1$ pivot (swap rows and columns $k$ and $r$, $s = 1$);
\ELSE
\STATE \textbf{(3)}: use $\left[ 
\begin{array}{cc}
a_{kk} & a_{rk} \\
a_{rk} & a_{rr}
\end{array}
\right]$ 
as a $2 \times 2$ pivot (swap rows and columns $k + 1$ and $r$, $s = 2$); 
\ENDIF
\end{algorithmic}
}}
\caption{Simplified Bunch-Kaufman partial pivoting strategy}\label{alg:SBKP}
\end{figure}

Let $\Pi = \widetilde{\Pi} \, \widehat{\Pi}^T$, we re-write equation \eqref{Eq:Perform LDL^T on Pi A Pi^T} as 
\[ \Pi^T \, A \, \Pi =  L \, D \, L^T,  \]
\noindent where 
\[L = \bordermatrix{
& s & n-s \cr
s & I_s & \cr
n-s & L_{21} & I_{n-s}
},\quad 
~D = \bordermatrix{
& s & n-s \cr
s & A_{11} &  \cr
n-s &  & A^{(1 + s)}
}.\]

To repeat this elimination process on the Schur complement $A^{(1+s)}$, we need a random projection on $A^{(1+s)}$ as well. While this random projection can be directly computed, it turns out to be unnecessary. Below we develop an efficient update formula to more efficiently compute it from $B$, the random projection on $A$. 

To this end, partition $\Omega \, \Pi = \left(
\begin{array}{cc}
\widetilde{\Omega}_1 & \widetilde{\Omega}_2 
\end{array}
\right)$ and $B \, \Pi = \left(
\begin{array}{cc}
\widetilde{B}_1 & \widetilde{B}_2 
\end{array}
\right)$. By equation \eqref{Eq:Apply Pi on A B Omega} and equation \eqref{Eq:Perform LDL^T on Pi A Pi^T},
\begin{equation*} 
\left(
\begin{array}{cc}
\widetilde{B}_1 & \widetilde{B}_2 
\end{array}
\right) = \left(
\begin{array}{cc}
\widetilde{\Omega}_1 & \widetilde{\Omega}_2 
\end{array}
\right) \left(
\begin{array}{cc}
I_s &  \\
L_{21} & I_{n-s}
\end{array}
\right)  \left(
\begin{array}{cc}
A_{11} &  \\
& A^{(1 + s)} 
\end{array}
\right)  \left(
\begin{array}{cc}
I_s & L_{21}^T \\
& I_{n-s}
\end{array}
\right), 
\end{equation*}
implying 
\[\widetilde{B}_1 = \left(\widetilde{\Omega}_1 + \widetilde{\Omega}_2 \, L_{21}\right) A_{11}, \quad \mbox{and} \quad \widetilde{B}_2 = \left(\widetilde{\Omega}_1 + \widetilde{\Omega}_2 \, L_{21}\right) A_{11} \, L_{21}^T + \widetilde{\Omega}_2 \, A^{(1 + s)} . \]

By choosing $\widetilde{\Omega}_2$ as the random matrix for a random projection on $A^{(1+s)}$, we immediately have 
\begin{align}\label{Eq:Update matrix B}
B^{(1+s)} & \stackrel{def}{=} \widetilde{\Omega}_2 \, A^{(1+s)} = \widetilde{B}_2 - \left( \widetilde{\Omega}_1  + \widetilde{\Omega}_2 \, L_{21} \right)A_{11} \, L_{21}^T  \notag\\
& = \widetilde{B}_2 - \widetilde{B}_1 \, L_{21}^T.
\end{align}

Directly computing $B^{(1+s)}$ as a matrix-matrix product $\widetilde{\Omega}_2 \, A^{(1+s)}$ would cost $O\left(pn^2\right)$ operations, whereas the update formula \eqref{Eq:Update matrix B} costs only $O\left(pn\right)$ operations, a largely negligible amount flop-wise. However, the random matrix $\widetilde{\Omega}_2$ is not de-correlated with $\Omega$, potentially making this new random projection less-effective. We will address this issue in Section \ref{Sec:Analysis of RCP}.

\begin{algorithm}
\caption{Blocked Randomized complete pivoting algorithm}
\label{alg:BRCP}
\begin{algorithmic}
\REQUIRE
\STATE symmetric matrix $A\in \mathbb{R}^{n \times n}$, block size $b$,  
\STATE pivots number $q$, sampling dimension $p ~(p \ge q)$.
\ENSURE
\STATE permutation matrix $\Pi$, unit lower triangular matrix $L$, 
\STATE and block diagonal matrix $D$ such that $\Pi A \Pi^T = L D L^T$. \\
\algrule
\STATE \textbf{initialize} $k = 1$; $\alpha = \frac{\sqrt{2}}{2}$; $\Pi = I_n$; $L = I_n$; 
\STATE \textbf{draw} random i.i.d. $\Omega \in \mathbb{R}^{p  \times n}$, compute $B = \Omega A$;
\WHILE {$k < n$}
\STATE $b = {\bf min}(b, ~ n - k + 1)$; $i = 1$; 
\WHILE{$i \le b$}
\IF {$q=1$}
\STATE \textbf{compute} $\pi_k = {\bf argmax}_{k \le j \le n} \| B( : , j) \|_2$;
\STATE \textbf{make} corresponding interchanges on $A,~B,~\Pi,~\mbox{and}~L$ when $k \neq \pi_k$;
\ELSIF {$i = 1$}
\STATE \textbf{perform} partial QRCP on $B(:,k:n)$ to find $q$ pivots;
\STATE \textbf{make} corresponding interchanges on $A,~B,~\Pi,~\mbox{and}~L$;
\ENDIF
\STATE \textbf{perform} SBKP pivoting (Figure \ref{alg:SBKP}) on $A(k : n, k : n)$ and obtain $s$;
\STATE \textbf{compute} $\widehat{k} = k + s - 1$;
\STATE \textbf{compute} $D(k : \widehat{k}, k : \widehat{k})$ and $L(\widehat{k} + 1 : n, k : \widehat{k})$ by using the same technique in Lapack routine SYTRF;
\STATE \textbf{compute} $k = k + s,~i = i + s$;
\IF {$k < n$}
\IF {$q=1$}
\STATE \textbf{compute} $B( : , k : n) ~ -= B( : , k-s : \widehat{k}) L(k : n, k-s : \widehat{k})^T$; 
\ELSIF {$i > b$}
\STATE $B( :, k : n) = B( :, k : n) - B( :, k-i+1 : k-1) L(k-i+1 : k-1,k-i+1 : k-1)^{-1} L(k : n, k-i+1 : k-1)^T$;    
\ENDIF
\ENDIF
\ENDWHILE
\STATE \textbf{compute} Schur complement $A(k : n, k : n)$;
\IF {$k = n$}
\STATE $D(k,k) = A(k,k)$; break;
\ENDIF
\ENDWHILE
\end{algorithmic} 
\end{algorithm}

\subsection{Implementations of the RCP algorithm}\label{Sec:Implementations of RCP}
In this section, we discuss two possible LAPACK-style \cite{anderson1999lapack} block implementations of the RCP algorithm with Algorithm \ref{alg:BRCP}. The RCP algorithm discussed above corresponds to the special case $b = q = 1$.

Parameter $b$ is the block-size, and parameter $q$ is the pivot number that is either $q = 1$ or $q = b$. In both cases, Algorithm \ref{alg:BRCP} delays updating the Schur complement of the matrix $A$ until $b$ elimination steps have been processed. In the inner while loop, Algorithm \ref{alg:BRCP} determines $s$ using SBKP strategy, computes factors $L$ and $D$, and then updates the remaining columns in $B$; and in the outer while loop, it updates the Schur complement using Level $3$ BLAS. For $q = 1$, Algorithm \ref{alg:BRCP} accumulates $b$ pivot columns in $A$ by repeating $b$ times the process of computing one step of QRCP on the $B$ matrix, performing one column pivot in the $A$ matrix, and updating remaining columns in $B$. For $q = b$, Algorithm \ref{alg:BRCP} chooses $b$ column pivots in $A$ simultaneously by computing a $b$-step partial QRCP on the $B$ matrix, performing $b$ column pivots in $A$, and updating remaining columns in $B$.

Algorithm \ref{alg:BRCP} performs the same BLAS-3 operations on the Schur complements with both $q=1$ and $q=b$. On the other hand,  Algorithm \ref{alg:BRCP} requires a random matrix $\Omega \in \mathbb{R}^{p \times n}$ with row dimension $p \ge q$. In our implementation, we typically choose $b = 64$  and $p = q + 4$ for good computational performance. Considering that Algorithm \ref{alg:BRCP} spends a total of $O(b p n)$ operations on the random projection matrix $B$, 
it would appear that $q = 1$ is a much better option. 

In Figure \ref{Fig:Run time and growth factor between three RCP}, we compare the efficiency difference between $q=1$ and $q=b$ on random matrices ({\bf Type $6$} in numerical experiments). It confirms our analysis above that $q=1$ is the better option. Hence in the numerical experiments (Section \ref{Sec:Numerical experiment}), we only use Algorithm \ref{alg:BRCP} with $q=1$. Curious enough, the option $q=1$ also gives smaller element growth factors than option $q=b$, perhaps due to the likely reason that finding one pivot each time is more reliable than finding $b$ pivots in each loop.

\begin{figure}[htbp]
\centering
\includegraphics[width=1.0\textwidth]{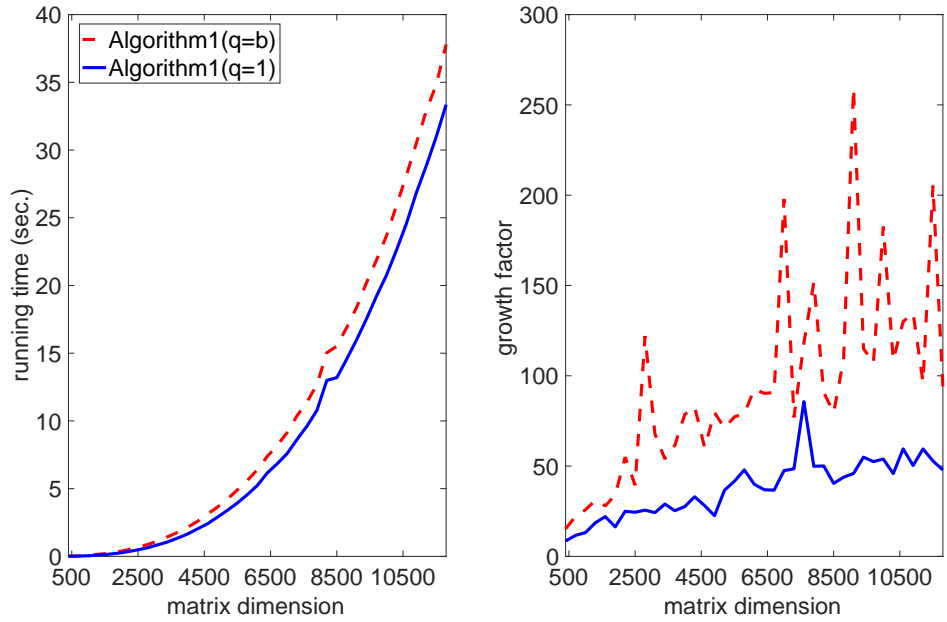}
\caption{run time and growth factor of Algorithm \ref{alg:BRCP} with different $q$}\label{Fig:Run time and growth factor between three RCP}
\end{figure}

\subsection{Operation counts}
In this section, we count the leading costs for the RCP algorithm, Bunch-Kaufman algorithm, and Aasen's algorithm. Here we just do analysis of the BLAS-2 versions of these algorithms, since same analysis results will also apply for the BLAS-3 versions. In practice, we only use the BLAS-3 versions since they are much more efficient than BLAS-2 versions in terms of communication costs.  

The operation count of the RCP algorithm is described in Figure \ref{Fig:Flowchart of flops for RCP}. The terms Mults, Divs, Adds and Comps stand for the number of multiplications, divisions, additions, and comparisons respectively for each step of the Algorithm \ref{alg:BRCP} with $b=q=1$. 

Table \ref{Tb:Flops of RCP BK and Aasen} gives flops upper bounds, for solving a symmetric $n\times n$ linear system of equations, by the RCP algorithm, Bunch-Kaufman algorithm, and Aasen's algorithm. All three algorithms are dominated by the same ${\displaystyle \frac{1}{3}n^3}$ multiplication and addition costs. Aasen's algorithm requires the fewest number of comparisons since it performs one column pivot per elimination step, whereas both the RCP and the Bunch-Kaufman algorithm perform up to two column pivots. Compared with the Bunch-Kaufman algorithm and Aasen's algorithm, the RCP algorithm performs additional $(2p-\frac{1}{2})n^2$ multiplications and additions. We choose $p = 5$ in the numerical experiments, so the additional work is $9.5 n^2$ flops, which is relatively insignificant for large $n$. In Section \ref{Sec:Numerical experiment}, the results of our numerical experiments match this analysis.

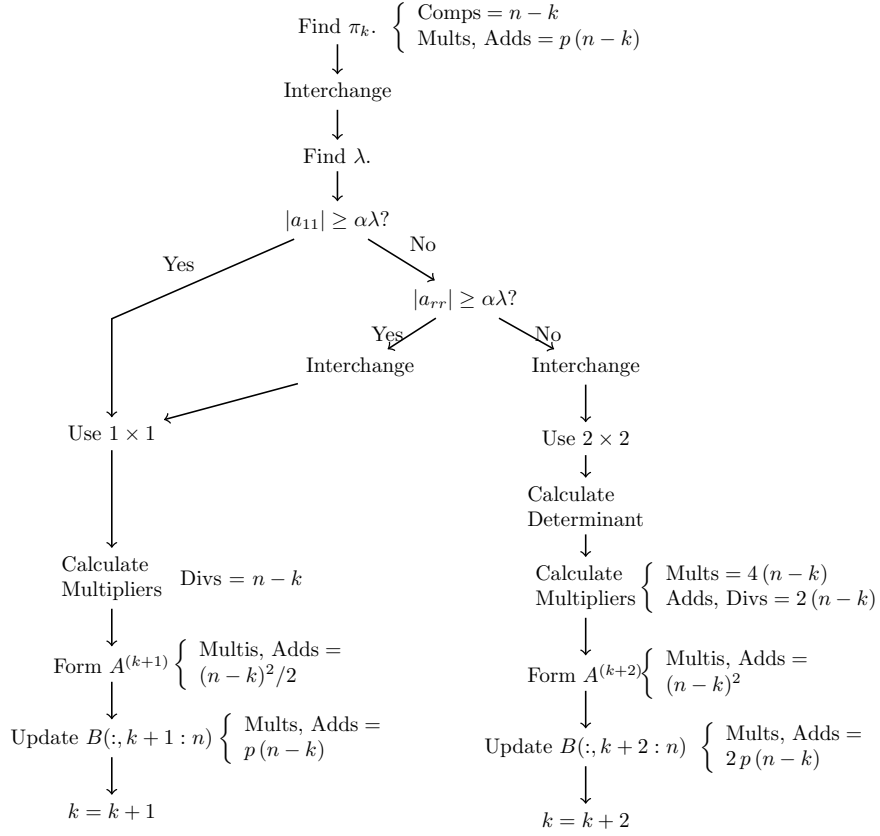
\begin{figure}[H]
\centering  
\begin{tikzpicture}[->,auto,semithick,align = left]
    \tikzset{
    term/.style = {right = 0.6cm}
    }
    \tikzstyle{every node} = [font = \small, scale = 0.8]
    \begin{scope}[node distance = 0.4cm, start chain = going below]
    \node[on chain] (s1) {Find $\pi_k$.};
    \node[term] at (s1) {$\left\{\begin{array}{l}
    \textrm{Comps} = n - k \\
    \textrm{Mults, Adds} = p \, (n - k) \\
    \end{array}
    \right.$};
    \node[on chain] (s2) {Interchange};
    \node[on chain] (s3) {Find $\lambda$.};
    \node[on chain] (s4) {$\left| a_{11} \right| \ge \alpha \lambda$?};
    \node[on chain] (s5) at ([shift = ({ - 3cm, - 0.8cm})] s4) {};
    \node[on chain] (s6) at ([shift = ({1.7cm, - 0.4cm})] s4) {$\left| a_{r r} \right| \ge \alpha \lambda$?};
    \end{scope}

    \begin{scope}[node distance = 0.5cm, start chain = going below]
    \node[on chain] (s7) at ([shift = ({ - 1.4cm, - 0.9cm})] s6) {Interchange};
    \node[on chain] (s9) [below = 1.2cm of s5] {Use $1 \times1$};

    \node[on chain,below = 1.3cm] (s11) {Calculate \\
    Multipliers};
    \node[right = 0.8cm] at (s11) {Divs = $n - k$};
    \node[on chain] (s14) {Form $A^{(k + 1)}$};                                                                                              
    \node[right = 0.7cm] at (s14) {$\left\{\begin{array}{l}
    \textrm{Multis, Adds} = \\
    (n - k)^2 / 2
    \end{array}\right.$};
    \node[on chain] (s16) {Update $B( : , k + 1 : n)$};
    \node[right = 1.3cm] at (s16) {$\left\{\begin{array}{l}
    \textrm{Mults, Adds} = \\
    p \, (n - k)
    \end{array}\right.$};

    \node[on chain] (s18) {$k = k + 1$};
    \end{scope}

    \begin{scope}[node distance = 0.5cm, start chain = going below]
    \node[on chain] (s8) at ([shift = ({1.6cm, - 0.9cm})]s6) {Interchange};
    \node[on chain] (s10) {Use $2 \times 2$};
    \node[on chain, below = 0.3cm] (s12) {Calculate \\
    Determinant};
    \node[on chain, below = 0.3cm] (s13) {Calculate \\
    Multipliers};
    \node[term] at (s13) {$\left\{\begin{array}{l}
    \textrm{Mults} = 4 \, (n - k) \\
    \textrm{Adds, Divs}=2 \, (n - k)
    \end{array}
    \right.$};
    \node[on chain] (s15) {Form $A^{(k + 2)}$};                                                                                              
    \node[term] at (s15) {$\left\{\begin{array}{l}
    \textrm{Multis, Adds} = \\
    (n - k)^2
    \end{array}\right.$};
    \node[on chain] (s17) {Update $B( : , k + 2 : n)$};
    \node[right = 1.4cm] at (s17) {$\left\{\begin{array}{l}
    \textrm{Mults, Adds} = \\
    2 \, p \, (n - k)
    \end{array}\right.$};
    \node[on chain] (s19) {$k = k + 2$};
    \end{scope}

    \path (s1) edge (s2);
    \path (s2) edge (s3);
    \path (s3) edge (s4);
    \path (s4) edge node  {No} (s6);
    \path (s5);\pgfgetlastxy{\sx}{\sy}
    \draw (s4) -- node [above left] {Yes}(\sx,\sy) -> (s9);
    \path (s6) edge node [left] {Yes} (s7);
    \path (s6) edge node [right] {No} (s8);
    \path (s8) edge (s10);
    \path (s7) edge (s9);
    \path (s9) edge (s11);
    \path (s10) edge (s12);
    \path (s12) edge (s13);
    \path (s11) edge (s14);
    \path (s13) edge (s15);
    \path (s14) edge (s16);
    \path (s15) edge (s17);
    \path (s16) edge (s18);
    \path (s17) edge (s19);
    \end{tikzpicture}

    \caption{Flowchart of operation counts for computing $A^{(k + s)}$ in RCP, $s = 1$ or $2$}  \label{Fig:Flowchart of flops for RCP}
    \end{figure}

    \begin{table}[htbp]
    \centering
    \caption{Comparison of the upper bounds for the number of operations}\label{Tb:Flops of RCP BK and Aasen}
    \begin{tabular}{|l|c|c|c|}
    \hline
  Algorithm & Mults and Divs & Adds & Comps \\
  \hline
  \multicolumn{1}{|l|}{\textbf{RCP}} & & &  \\
  \quad decomposition & $\frac{1}{6} n^3 + (p + 1) n^2 $  & $\frac{1}{6} n^3 + (p - \frac{1}{2}) n^2 $ & $n^2$  \\[3pt]
  \quad solution phase &  $n^2 $ & $n^2 $  &  0 \\[3pt]
  \quad total & $\frac{1}{6} n^3 + (p + 2) n^2 $  & $\frac{1}{6} n^3 + (p + \frac{1}{2}) n^2 $ & $n^2$  \\[3pt]
  \hline 
  \multicolumn{1}{|l|}{\textbf{Bunch-Kaufman}} & & &  \\
  \quad decomposition & $\frac{1}{6} n^3 + \frac{3}{4} n^2 $ & $\frac{1}{6} n^3 + \frac{1}{4} n^2 $ & $n^2 $ \\[3pt]
  \quad solution phase & $n^2 $  & $n^2 $ & 0  \\[3pt]
  \quad total & $\frac{1}{6} n^3 + \frac{7}{4} n^2 $  & $\frac{1}{6} n^3 + \frac{5}{4} n^2 $  & $n^2 $  \\[3pt]
  \hline 
  \multicolumn{1}{|l|}{\textbf{Aasen}} & & &  \\
  \quad decomposition & $\frac{1}{6} n^3 + 2 n^2 $ & $\frac{1}{6} n^3 +  n^2 $ & $\frac{1}{2} n^2 $ \\[3pt]
  \quad solution phase & $n^2 $  & $n^2 $ & n  \\[3pt]
  \quad total & $\frac{1}{6} n^3 + 3 n^2 $  & $\frac{1}{6} n^3 + 2 n^2 $  & $\frac{1}{2} n^2 $  \\[3pt]
  \hline
  \end{tabular}
  \end{table}

\section{Analysis of randomized complete pivoting algorithm}\label{Sec:Analysis of RCP}
To demonstrate its numerical reliability, in this section, we develop probabilistic upper bounds on both entries of the lower triangular matrix $L$ and on the element growth factor for the RCP algorithm. These upper bounds are much stronger than those for the Bunch-Kaufman algorithm. As the success of the RCP algorithm also depends on the numerical reliability of the random projections, we will study their numerical stability issues as well. 

\subsection{Reliability of randomized column pivots}
The RCP algorithm chooses column pivots using norm-based column selections on the correlated random projections. Lemma \ref{Le:Randomized Norm Preservation} below justifies the reliability of these pivots and is the basis of much of our analysis on the RCP algorithm. In order to prove Lemma \ref{Le:Randomized Norm Preservation}, we first show that the random projections do satisfy the Johnson-Lindenstrauss Lemma in a conditional sense, and then remove the condition with the Law of Total Probabilities. 

\begin{lemma}\label{Le:Randomized Norm Preservation}
(Randomized Norm Preservation) 
Given $A \in \mathbb{R}^{n \times n}$, over-sampling parameter $p > 0$, and $0 < \epsilon, ~\delta < 1$. Assume that for every $k$ for which the RCP algorithm performs a random projection, the random matrix $\Omega^{(k)}$ is chosen as in equation \eqref{Eq:Update matrix B} and let $\pi_k = {\bf argmax}_{k \le j \le n} \left\| \Omega^{(k)} A^{(k)}( : , j) \right\|_2$ be the column pivot. Then 
\begin{equation}\label{Eq:Randomized Norm Preservation}
\mathbb{P} \left( \left\| A^{(k)}( : , \pi_k) \right\|_2  \ge \sqrt{ \frac{1 - \epsilon}{1 + \epsilon} } \,  \left\| A^{(k)} \right\|_{1,2} \right) \ge 1- \delta
\end{equation}
holds for all such $k$, provided that $p \ge \frac{4}{\epsilon^2 - \epsilon^3} \ln{ \left( \frac{n(n + 1)}{\delta} \right) } $.
\end{lemma}

\begin{proof} The diagonal permutation RCP chooses in the $LDL^T$ factorization is a discrete random variable. There are two parts in the RCP pivoting scheme. At each step $k$, RCP chooses a column pivot based on column selection in the random projection matrix, and then makes a deterministic symmetric pivot based on the SBKP strategy. The $LDL^T$ factorization is uniquely determined by the diagonal permutation $\Pi$ and block-diagonal pattern $\Delta$, which indicates whether each diagonal entry in $D$ is a $1 \times 1$ pivot or part of a $2 \times 2$ pivot. 

Each $LDL^T$ factorization process with RCP results in a diagonal permutation and block-diagonal pattern pair $\left(\Pi, \Delta\right)$, with some matrix $A$ dependent probability $\mathbb{P}\left(\Pi, \Delta\right)$. It is clear that these $\left(\Pi, \Delta\right)$ pairs are mutually exclusive and 
satisfy
\begin{equation}\label{Eqn:prob}
\sum_{\left(\Pi, \Delta\right)} \mathbb{P}\left(\Pi, \Delta\right) = 1. 
\end{equation}
Let $\Omega\in \mathbb{R}^{p \times n}$ be the I.I.D. Gaussian matrix drawn by RCP. 

RCP begins its first column pivot with $k = 1$. For $1 \le j \le n$, we denote the event that $ A( : , j)$ satisfies $\epsilon$-JL condition by $E_{1, j}$. Thus 
\begin{equation}\label{Eqn:E1}
 E_1 \stackrel{def}{=} \bigcap_{1 \le j \le n} E_{1, j}
\end{equation}
is the event where every column of $A^{(1)}=A$ satisfies the  $\epsilon$-JL condition under the random matrix $\Omega$. For any diagonal permutation $\Pi$ and block-diagonal pattern pair $\left(\Pi, \Delta\right) $, the column pivot $\pi_1$ in the conditional event $E_1  \left| \right. \left(\Pi, \Delta\right) $ must satisfy for any $1 \le j \le n$,
\begin{equation}\label{eq:condition between B and A}
\sqrt{1 - \epsilon} \left\| A^{(1)} ( : , j) \right\|_2  \le \left\| \frac{1}{\sqrt{p}} B( : , j) \right\|_2  \le \left\| \frac{1}{\sqrt{p}} B( : , \pi_1) \right\|_2 \le \sqrt{1 + \epsilon} \left\| A^{(1)} ( : , \pi_1) \right\|_2,  
\end{equation}
which implies 
\begin{equation}\label{Eqn:Cond1}
 \left\| A^{(1)}( : , \pi_1) \right\|_2  \ge \sqrt{ \frac{1 - \epsilon}{1 + \epsilon} } \,  \left\| A^{(1)} \right\|_{1,2}.
\end{equation}

Under the diagonal permutation and block-diagonal pattern pair $\left(\Pi, \Delta\right)$, we further assume that RCP performs a random projection at $k$-th elimination step under $\left(\Pi, \Delta\right)$ for some $k$. Partition $\Omega$ as in equation \eqref{Eq:Update matrix B}. It is clear that $\widetilde{\Omega}_2$ (the last $n-s$ rows of $\widetilde{\Omega}$) is an I.I.D. Gaussian matrix conditioned under permutation $\Pi$. 

Let $A^{(k)}$ be the Schur Complement. As with the case $k = 1$, we denote the event that $A^{(k)}( : , j)$ satisfies $\epsilon$-JL condition by $E_{k, j}$, and let 
\begin{equation}\label{Eqn:Ek}
 E_k \stackrel{def}{=} \bigcap_{k \le j \le n} E_{k, j}
\end{equation}
be the event that all columns of $A^{(k)}$ satisfy the $\epsilon$-JL condition. As in Equation \eqref{Eqn:Cond1}, the column pivot $\pi_k$ in the conditional event $E_k  \left| \right. \left(\Pi, \Delta\right) $ must satisfy the condition
\begin{equation}\label{Eqn:Condk}
 \left\| A^{(k)}( : , \pi_k) \right\|_2  \ge \sqrt{ \frac{1 - \epsilon}{1 + \epsilon} } \,  \left\| A^{(k)} \right\|_{1,2}.
\end{equation}

In case RCP does not perform a random projection at $k$-th elimination step, we define $E_{k, j}$ and $E_k$ to be identically true events with $\mathbb{P}\left(E_{k, j}\right) = \mathbb{P}\left(E_{k}\right) = 1$.  

To develop a lower bound on the probability $\mathbb{P} (E_k)$, we appeal to Law of Total Probability (Theorem \ref{Th:Law of Total Probability}), 

$$
\mathbb{P} (E_k) = \sum_{\left(\Pi, \Delta\right)} \mathbb{P} \left( E_k \left| \right. \left(\Pi, \Delta\right) \right) \mathbb{P} \left(\Pi, \Delta\right), 
$$
where the sum is over all possible permutation and block-diagonal pattern pairs. By the Johnson-Lindenstrauss Lemma (Lemma \ref{Le:JL lemma}),  
\begin{align*}
\mathbb{P} \left( E_k \left| \right. \left(\Pi, \Delta\right) \right) & = \mathbb{P} \left( \bigcap_{k \le j \le n} E_{k, j} \left| \right. \left(\Pi, \Delta\right) \right)  \ge \sum_{j = k}^n \mathbb{P}\left(E_{k,j} | \left(\Pi, \Delta\right) \right) - (n - k) \\
& \ge \sum_{j = k}^n \left(1 - 2 \exp \left( - \frac{ (\epsilon^2 - \epsilon^3) p }{4} \right)\right) - (n - k) \\
& = 1- 2(n - k + 1) \exp \left( - \frac{ (\epsilon^2 - \epsilon^3) p }{4} \right).
\end{align*}

It now follows from Equation \eqref{Eqn:prob} that  
\begin{align*}
\mathbb{P}(E_k) & \ge \left( 1 - 2 (n - k + 1) \exp \left( - \frac{ (\epsilon^2 - \epsilon^3) p }{4} \right) \right) \sum_{\left(\Pi, \Delta\right)} \mathbb{P}\left(\Pi, \Delta\right) \\
& = 1- 2 (n - k + 1) \exp \left( - \frac{ (\epsilon^2 - \epsilon^3) p }{4} \right).
\end{align*}

Thus, for large enough $p$, the norm of the pivot column  chosen by RCP is a factor of ${\displaystyle \sqrt{ \frac{1 - \epsilon}{1 + \epsilon}}}$ away from optimal for any given $k$ with high probability. 

To demonstrate that RCP chooses all such pivot columns with high probability, we let  
\[{\displaystyle E \stackrel{def}{=} \bigcap_{1 \le k \le n} E_k }\]
to be the event that all the column pivots chosen by RCP satisfy the $\epsilon$-JL condition. It follows that 
 \begin{align*}
\mathbb{P} \left( E \right) & = \mathbb{P} \left( \bigcap_{k = 1}^n E_k \right) \ge \sum_{k = 1}^n \mathbb{P}(E_k) - (n - 1)  \\
& \ge \sum_{k = 1}^n \left( 1 - 2 (n - k + 1) \exp \left( - \frac{ (\epsilon^2 - \epsilon^3) p }{4} \right) \right) - (n - 1)  \\
& \ge 1 - n (n + 1) \exp \left( - \frac{ (\epsilon^2 - \epsilon^3) p }{4} \right), 
\end{align*}
the last expression is bounded below by $1 - \delta$ for $p \ge \frac{4}{\epsilon^2 - \epsilon^3} \ln{ \left( \frac{n (n + 1)}{\delta} \right) }$.  
\end{proof}

Lemma \ref{Le:Randomized Norm Preservation} suggests that $p$ needs only to grow logarithmically with $n$ and $1/\delta$ for reliable pivots. For illustration, if we choose $n = 1000, \epsilon = 0.5, \delta = 0.05$, then $p \ge 538$. However, in practice, $p$ value like $5$ suffices for RCP to obtain reliable pivots and hence reliable block $LDL^T$ factorization.

\subsection{Upper Bounds on entries of the L matrix} \label{Sec:Analysis of bounds for L}
In addition to potentially exponential element growth, the Bunch-Kaufman algorithm can also become numerically unstable 
due to potentially unbounded entries in the computed $L$ matrix~\cite{higham1997stability}. In contrast, the entries in the $L$ matrix are upper bounded by $O(1)$ with symmetric complete pivoting~\cite{bunch1971direct}. In this section, we show that, with high probability, the entries of the matrix $L$ are upper bounded by $O\left(\sqrt{n}\right)$ with the RCP algorithm, thus avoiding the potential numerical instability caused by unbounded $L$ entries. 

\begin{theorem}\label{Thm:Upper Bounds on L entries}
Given $A \in \mathbb{R}^{n \times n}$, $0 < \epsilon, ~\delta < 1$, and over-sampling parameter 
\[{\displaystyle p \ge \frac{4}{\epsilon^2 - \epsilon^3} \ln{ \left( \frac{n(n + 1)}{\delta} \right) },} \]
then with probability at least $1 -\delta$,  
\begin{equation}\label{Eqn:UpperBoundL}
 {\displaystyle \left| l_{i,k} \right| \le \frac{1 + \sqrt{\frac{1 + \epsilon}{1 - \epsilon}}\, \sqrt{n - k + 1}}{{\bf min}\left(\alpha, 1 - \alpha^2\right)}} \quad \mbox{for all}\quad 1 \le k < i \le n. 
 \end{equation}
\end{theorem}

\begin{proof} There are three non-trivial cases when the SBKP strategy (Figure \ref{alg:SBKP}) is applied to the first column of matrix $A \stackrel{def}{=} A^{(1)}$ with $k = 1$, resulting in an $s \times s$ pivoting block with $s = 1$ or $2$. The SBKP strategy will then proceed to perform subsequent eliminations on the Schur complement $A^{(1 + s)}\stackrel{def}{=}\left(a_{i j}^{(1+s)}\right)_{i,j = k+s}^n$. Below we consider each case. 

\begin{description}
\item[{\bf Case (1):}] $a_{1,1}$ is the $1 \times 1$ pivot with $\left| a_{1,1} \right| \ge \alpha \lambda$, where $\lambda = {\bf max}_{2 \le j \le n} \left|a_{j,1}\right|$. In this case, 
\begin{align*}
l_{j, 1} & = a_{j, 1} / a_{1,1}, \quad \mbox{for}\quad j > 1, \end{align*}
\begin{align*}
\mbox{and therefore} \quad \left| l_{j, 1} \right|  & \le \frac{1}{\alpha},
\end{align*}
which obviously satisfies equation \eqref{Eqn:UpperBoundL}. 
\item[{\bf Case (2):}] $a_{r, r}$ is the $1 \times 1$ pivot with $\left| a_{r r} \right| \ge \alpha \lambda$. After the row and column interchanges, we have 
\begin{equation}\label{Eqn:lj1}
l_{j, 1} = \left\{\begin{array}{ll} 
a_{j, r} / a_{r, r}, & \mbox{for}\quad 2 \le j \le n, ~j \not= r, \\
a_{1, r} / a_{r, r}, & \mbox{for}\quad j = r.\end{array}\right.
\end{equation}

For any $j \neq r$, it follows from Lemma \ref{Le:Randomized Norm Preservation} that 
$$
|a_{j, r} | \le \left\| A( : , r) \right\|_2 \le \sqrt{\frac{1 + \epsilon}{1 - \epsilon}}\, \left\|A( : , 1) \right\|_2 \le \sqrt{\frac{1 + \epsilon}{1 - \epsilon}}\, \lambda \, \sqrt{n}. 
$$
Plugging this into equation \eqref{Eqn:lj1} yields 
\[{\displaystyle \left| l_{j, 1} \right|  \le \frac{\sqrt{\frac{1 + \epsilon}{1 - \epsilon}} \, \sqrt{n}}{\alpha}}, \]
which satisfies equation \eqref{Eqn:UpperBoundL} with $k=1$. 
\item[{\bf Case (3):}] The SBKP strategy chooses a $2 \times 2$ pivot $\left( 
\begin{array}{cc}
a_{1,1} & a_{r,1} \\
a_{r,1} & a_{r,r}
\end{array}
\right)$ with $s = 2$. After the row and column interchanges, we have that for any $j \ge 3$, let 
\[ 
i = \left\{\begin{array}{ll} 
j, & \mbox{for}\quad 3 \le j \le n, ~j \not= r, \\
2, & \mbox{for}\quad j = r.\end{array}\right.\]
Then 
\begin{align*}
\left(l_{j, 1}, \;\; l_{j, 2}\right) &= \left(a_{i, 1}, \;\; a_{i, r}\right)  \left(
\begin{array}{cc}
a_{1,1} & a_{r,1} \\
a_{r,1} & a_{r,r}
\end{array}
\right)^{-1} \\
&= \frac{1}{a_{1,1} \, a_{r, r} - a_{r,1}^2}
\left(a_{i, 1} \, a_{r, r} - a_{r, 1} \, a_{i, r}, \;\; a_{1,1} \, a_{i, r} - a_{r, 1} \, a_{i, 1}\right).
\end{align*}
\noindent Since $| a_{1, 1} | < \alpha \lambda$, $| a_{r, r} | < \alpha \lambda$, and
\[| a_{i, r} | \le \sqrt{\frac{1 + \epsilon}{1 - \epsilon}}\, \lambda \, \sqrt{n}, \]
we deduce that 
$$
| l_{j, 1} | ,\;\; | l_{j, 2} | \le \frac{\lambda^2 +  \sqrt{\frac{1 + \epsilon}{1 - \epsilon}}\, \lambda^2 \, \sqrt{n}}{\lambda^2 \, (1 - \alpha^2)} = \frac{1 +  \sqrt{\frac{1 + \epsilon}{1 - \epsilon}}\, \sqrt{n} }{1 - \alpha^2}, 
$$
which again satisfies equation \eqref{Eqn:UpperBoundL} with $k=1$.
\end{description} 

To complete the proof for Theorem \ref{Thm:Upper Bounds on L entries}, we recursively apply the same argument above to the Schur complement $A^{(k)}$ in each case, taking note that the dimension of $A^{(k)}$ for any $k$ is $n-k+1$. 
\end{proof}

In next section, we will show that we can also deduce a much better column growth factor bound using Wilkinson's techniques for GECP \cite{wilkinson1961error}. His proof is dependent on the fact that the pivots in GECP are maximum elements in the Schur complement. Our pivots are not necessarily maximum elements, but they are closely related to the maximum column norm in the Schur complement. Thus, we need more subtle analysis.

\subsection{Probability analysis on column growth factor}\label{Sec:Analysis of growth factor of RCP}
In this section, we present a rigorous upper bound on the column norm growth factor of $L D L^T$ factorization computed by the RCP algorithm. 

\begin{theorem}\label{Th:Column Growth Factor for RCP}
(Column Growth Factor for Randomized Complete Pivoting).\\ 
Given $A \in \mathbb{R}^{n \times n}$, $0 < \epsilon, ~\delta < 1$, and over-sampling parameter 
\[{\displaystyle p \ge \frac{4}{\epsilon^2 - \epsilon^3} \ln{ \left( \frac{n(n + 1)}{\delta} \right) },} \]
then the column norm growth factor of RCP algorithm, with $\alpha = \frac{\sqrt{2}}{2}$ in the SBKP strategy, satisfies
\begin{equation*}
   \rho_{\textrm{col}} \le \left(\sqrt{\frac{2(1 + \epsilon)}{1 - \epsilon}}\right)^{3 + \ln(n-1)} \, \left(\sqrt{n+2}\right)^{2 + \ln(n-1)} 
   \end{equation*}
with probability at least $1-\delta$.
\end{theorem}

\begin{proof} Following Wilkinson \cite{wilkinson1961error}, our approach is to establish a recursive relationship among the diagonal entries of the $D$ matrix using the Hadamard's inequality (Theorem \ref{Th:Hadamard Inequality}). The upper bound on $\rho_{\textrm{col}}$ is then obtained by solving this recursion. 

Similar to the element growth factor analysis for Bunch-Parlett algorithm \cite{bunch1971analysis}, we adopt the following notation: 
\begin{equation*}
pivot[k] \stackrel{def}{=}
\left\{ 
\begin{array}{lll}
1, \:1 \times 1 ~\mbox{pivot} ~ \mbox{in} ~ A^{(k)}, \left( ~ \left| a_{k k}^{(k)} \right| \ge \alpha \, \lambda^{(k)} ~ \mbox{or} ~ \left| a_{r r}^{(k)} \right| \ge \alpha \, \lambda^{(k)} > \left| a_{k k}^{(k)} \right| \right), \\
\\
2, \:2 \times 2 ~\mbox{pivot} ~ \mbox{in} ~ A^{(k)}, \left( ~  \left| a_{k k}^{(k)} \right|, \left| a_{r r}^{(k)} \right| < \alpha \, \lambda^{(k)} \right),  \\
\\
0, \:2 \times 2 ~\mbox{pivot}~ \mbox{in} ~ A^{(k - 1)}, \left( ~ \left| a_{k-1, k-1}^{(k - 1)} \right|, \left| a_{r r}^{(k - 1)} \right| < \alpha \, \lambda^{(k-1)}\right).
\end{array}
\right.
\end{equation*}

$$v^{(k)} \stackrel{def}{=} \left| \det \left( \begin{array}{cc}
a_{k k}^{(k)} & a_{r k}^{(k)} \\
a_{r k}^{(k)} & a_{r r}^{(k)}
\end{array}
\right) \right| = \left| a_{k k}^{(k)} \, a_{r r}^{(k)} - \left( a_{r k}^{(k)} \right)^2 \right|.$$

$$
p_k \stackrel{def}{=} \left\{ \begin{array}{ll}
\mbox{$\left|a_{k k}^{(k)}\right|$ or $\left|a_{r r}^{(k)}\right|$} & \mbox{if $pivot[k] = 1,$} \\
\sqrt{v^{(k)}} & \mbox{if $pivot[k] = 2,$} \\
\sqrt{v^{(k-1)}} & \mbox{if $pivot[k] = 0.$}
\end{array}
\right.
$$

Therefore, for $k \le m \le n$,
\begin{equation}\label{Eq:determinant value of A}
\left| \det \left( A^{(k)}(k : m, k : m) \right) \right| = \prod_{j = k}^m p_j.
\end{equation}

Let $c_k \stackrel{def}{=} \left\| A^{(k)} ( : , \pi_k) \right\|_2$ be the 2-norm of the pivot column of $A^{(k)}$, where $\pi_k = {\bf argmax}_{k \le j \le n} \left\| \Omega^{(k)} A^{(k)}( : , j) \right\|_2$. When $pivot[k] = 0$, $c_k = c_{k - 1}$. For each pivot possibility, we deduce a lower bound for $p_k$.

{\bf Case (1):}
$pivot[k] = 1$.
$$
p_k \ge {\bf max}\left\{\alpha \, \lambda^{(k)}, \left|a_{kk}^{(k)}\right|\right\} \ge {\bf max}\left\{\alpha \, \lambda^{(k)}, \alpha \, \left|a_{kk}^{(k)}\right|\right\} \ge \frac{\alpha}{\sqrt{n-k+1}} \, c_k,
$$
where the first inequality holds because of the definition of $p_k$ when $pivot[k] = 1$, and the second inequality holds because of the fact that $0 < \alpha < 1$, and the third inequality holds because of a simple property of the entry with the largest magnitude in the pivot column of $A^{(k)}$.

{\bf Case (2):}
$pivot[k] = 2$.
$$
p_k = \sqrt{v^{(k)}} 
= \sqrt{\left| \left( a_{r k}^{(k)} \right)^2 - a_{k k}^{(k)} \, a_{r r}^{(k)} \right|}
\ge \sqrt{1 - \alpha^2} \, \lambda^{(k)} \ge \frac{\sqrt{1 - \alpha^2}}{\sqrt{n - k + 1}} \, c_k.
$$
where the first inequality holds because $\left|a_{r k}^{(k)}\right| = \lambda^{(k)}, \left|a_{k k}^{(k)}\right| < \alpha \, \lambda^{(k)} \mbox{ and } \left|a_{r r}^{(k)}\right| < \alpha \, \lambda^{(k)}.$

{\bf Case (3):}
$pivot[k] = 0$.
\begin{eqnarray*}
p_k &=& \sqrt{v^{(k - 1)}} 
= \sqrt{\left| \left( a_{r,k-1}^{(k-1)} \right)^2 - a_{k-1 ,k-1}^{(k-1)} a_{r r}^{(k-1)} \right|} \\
&\ge& \frac{\sqrt{1 - \alpha^2}}{\sqrt{n - k + 2}} c_{k-1} = \frac{\sqrt{1 - \alpha^2}}{\sqrt{n - k + 2}} c_k.
\end{eqnarray*}
since $\left|a_{r,k-1}^{(k-1)}\right| = \lambda^{(k-1)}, \left|a_{k-1,k-1}^{(k-1)}\right| < \alpha \, \lambda^{(k-1)} \mbox{ and } \left|a_{r r}^{(k-1)}\right| < \alpha \, \lambda^{(k-1)}.$

Let $t(\alpha) \stackrel{def}{=} {\bf min}\left(\alpha, \; \sqrt{1 - \alpha^2} 
\right)$, then all the above three cases deduce to  
$$
p_k \ge \frac{t(\alpha)}{\sqrt{n - k + 2}} \, c_k.
$$ 

From equation \eqref{Eq:determinant value of A},
\begin{align}\label{Eq:lower bound of column growth factor}
\left| \det \left( A^{(k)}(k : m, k : m) \right) \right|
\ge \prod_{j = k}^m \frac{t(\alpha)}{\sqrt{n - j + 2}} \, c_j = t(\alpha)^{m - k + 1} \, \prod_{j = k}^m \frac{c_j}{\sqrt{n - j + 2}}.
\end{align}

We can also obtain an upper bound of the determinant of $A^{(k)}(k : m, k : m)$ by applying Hadamard's inequality (Theorem \ref{Th:Hadamard Inequality}), that is 
\begin{align}\label{Eq:upper bound of column growth factor}
\left| \det \left( A^{(k)}(k : m, k : m) \right) \right| & \le \prod_{j = k}^m \left\| A^{(k)} (k : m, j) \right\|_2 \le \prod_{j = k}^m \left\| A^{(k)}( : , j) \right\|_2   \notag \\
& \le \left( \sqrt{\frac{1 + \epsilon}{1 - \epsilon}} \right)^{m - k + 1} \left\| A^{(k)} ( : , \pi_k) \right\|_2^{m - k + 1} \notag \\
& = \left( \sqrt{\frac{1 + \epsilon}{1 - \epsilon}} \right)^{m - k + 1} \, c_k^{m - k + 1},
\end{align}
where the second inequality holds because the $2$-norm of the $j$-th column of the Schur complement $A^{(k)}$ is greater than or equal to the $2$-norm of the first few entries of the $j$-th column of the Schur complement $A^{(k)}$. The third one holds because of Lemma \ref{Le:Randomized Norm Preservation}. Combining inequalities \eqref{Eq:lower bound of column growth factor} and \eqref{Eq:upper bound of column growth factor} together for $\left| \det\left(A^{(k)} (k : m, k : m) \right) \right|$, we obtain that 
\begin{align*}
\prod_{j = k}^m \frac{c_j}{\sqrt{n - j + 2}} \le \left( \frac{1}{t(\alpha)}\, \sqrt{\frac{1 + \epsilon}{1 - \epsilon}} \right)^{m - k + 1} c_k^{m - k + 1}.
\end{align*}

For all $1\le k \le m \le n$, we take logarithms on both sides and define
$$
q_j \stackrel{def}{=} \ln(c_j),
$$
the above inequality becomes a linear recursive inequality, 
$$
 \sum_{j = k + 1}^m q_j \le (m - k + 1) \, \ln \left(\frac{1}{t(\alpha)}\,  \sqrt{\frac{1 + \epsilon}{1 - \epsilon}} \right) + \sum_{j = k}^m \ln \left( \sqrt{n - j + 2} \right) + (m - k) \, q_k.
$$
\noindent Below we derive an upper bound on $q_m$ by solving this recursion. 

Dividing both sides by $(m - k)(m - k + 1)$ for $k = 2, 3, \cdots, m - 1$, and by $m-k$ for $k=1$ and adding all $m-1$ equations, on observing that
$$
\frac{1}{m - 1} + \frac{1}{(m - 2)(m - 1)} + \cdots + \frac{1}{(m - k)(m - k + 1)} = \frac{1}{m - k},
$$
we obtain
\begin{eqnarray}
q_m  & \le & \left( 1 + \sum_{j = 1}^{m - 1} \frac{1}{j}  \right) \ln \left(\frac{1}{t(\alpha)}\, \sqrt{\frac{1 + \epsilon}{1 - \epsilon}} \right)  \notag    \\
&& + \ln \left( \sqrt{n - m + 2} \right) + \sum_{j = 1}^{m-1} \frac{1}{j} \, \ln \left( \sqrt{n - m + 2 + j} \right) + q_1 \nonumber \\
& \leq & \left( 1 + \sum_{j = 1}^{m - 1} \frac{1}{j}  \right) \left(\ln \left(\frac{1}{t(\alpha)}\, \sqrt{\frac{1 + \epsilon}{1 - \epsilon}} \right) + \ln \sqrt{n +1}\right) + \ln \left( \sqrt{n + 2} \right) + q_1\label{Eq:Inequality between qm and q1} 
\end{eqnarray}
\noindent Since
\begin{equation}\label{Eq:Upper bound of reciprocal sum}
\sum_{j = 1}^{m - 1}\frac{1}{j} \le 1 + \ln(m - 1),
\end{equation}
Plugging this into \eqref{Eq:Inequality between qm and q1}, 
\begin{align*}
q_m - q_1 & \le \left( 2 + \ln(n - 1) \right) \left(\ln \left( \frac{1}{t(\alpha)}\,\sqrt{\frac{1 + \epsilon}{1 - \epsilon}} \right) + \ln \sqrt{n +1}\right) + \ln \left( \sqrt{n + 2} \right).
\end{align*}

The upper bound above holds true for all $m$. Taking the exponential of both sides, for all $1\le m\le n$,
\begin{equation*}
\frac{c_m}{c_1} \le \sqrt{n + 2}\, \left( \frac{1}{t(\alpha)} \,\sqrt{\frac{1 + \epsilon}{1 - \epsilon}}\right)^{\left( 2 + \ln(n - 1) \right)} \,  
\left(\sqrt{n+1}\right)^{\left( 2 + \ln(n - 1) \right)}.
\end{equation*}

Now we estimate $\rho_{\textrm{col}}$ using ${\displaystyle\frac{c_m}{c_1}}$,
\begin{align}\label{Eq:Compute column growth factor}
\rho_{\textrm{col}} & = \frac{{\bf max}_m \left\| A^{(m)} \right\|_{1,2}}{\| A \|_{1,2}}  \le \sqrt{\frac{1 + \epsilon}{1 - \epsilon}} \, {\bf max}_m \frac{ \left\| A^{(m)}( : , \pi_m) \right\|_2}{\| A^{(1)} ( : , \pi_1) \|_2} =  \sqrt{\frac{1 + \epsilon}{1 - \epsilon}} \, {\bf max}_m \frac{c_m}{c_1} \notag \\
& \leq \left( \frac{1}{t(\alpha)} \,\sqrt{\frac{1 + \epsilon}{1 - \epsilon}}\right)^{\left(3 + \ln(n - 1) \right)} \,  
\left(\sqrt{n+2}\right)^{\left(2 + \ln(n - 1) \right)}.
\end{align}

The function $t(\alpha)$ achieves its maximum at $\alpha = \frac{\sqrt{2}}{2}$. With this choice of $\alpha$, equation \eqref{Eq:Compute column growth factor} becomes
\begin{align} \label{Eq: Column growth factor bound}
\rho_{\textrm{col}} & \le \left(\sqrt{\frac{2(1 + \epsilon)}{1 - \epsilon}}\right)^{\left(3 + \ln(n - 1) \right)} \,  
\left(\sqrt{n+2}\right)^{\left(2 + \ln(n - 1) \right)}.
\end{align}
\end{proof}

Corollary \ref{Cor:Element Growth Factor for RCP} below follows directly from Lemma \ref{Le:Inequality between element GF and column norm GF} and Theorem \ref{Th:Column Growth Factor for RCP}.

\begin{corollary}\label{Cor:Element Growth Factor for RCP}
(Element Growth Factor for Randomized Complete Pivoting). Under the assumptions of Theorem \ref{Th:Column Growth Factor for RCP}, the element growth factor of RCP algorithm, with $\alpha = \frac{\sqrt{2}}{2}$ in the SBKP strategy, satisfies
\begin{equation*}
   \rho_{\textrm{elem}} \le \sqrt{n}\, \left(\frac{2(1 + \epsilon)}{1 - \epsilon}\right)^{3 + \ln(n-1)} \, \left(\sqrt{n+2}\right)^{2 + \ln(n-1)}. 
   \end{equation*}
with probability at least $1-\delta$. 
\end{corollary}

To compare our upper bound on $\rho_{\textrm{elem}}$ with 
element growth upper bound for GECP on a general $n\times n$ non-symmetric matrix, recall from Wilkinson \cite{wilkinson1961error} that 
\begin{equation*}
\rho^{\textrm{gecp}}_{\textrm{elem}} \le \sqrt{n} f(n) = \sqrt{n} \left( 2 \cdot 3^{\frac{1}{2}} \cdots n^{\frac{1}{n - 1}} \right)^{1/2} \le 2 n^{1/2} n^{\frac{1}{4} \ln(n)}.
\end{equation*}
\noindent While our upper bound on $\rho_{\textrm{elem}}$ is much larger than that for $\rho^{\textrm{gecp}}_{\textrm{elem}}$, they are comparable in that the dominant factor in both upper bounds is $n^{\ln(n)}$, which for large $n$ is much less than $2^{n-1}$, the attainable element growth upper bound for Gaussian elimination with partial pivoting. For diagonal pivoting methods, Bunch and Parlett \cite{bunch1971direct} proved that the element growth factor is bounded above by $(2.57)^{n-1}$ for their algorithm. In 1977, Bunch and Kaufman \cite{bunch1977some} proposed a partial pivoting strategy to solve symmetric indefinite linear systems, and showed that the element growth factor is also bounded above by $(2.57)^{n - 1}$. On the other hand, Bunch \cite{bunch1971analysis} gave an element growth bound $3nf(n)$ for the Bunch-Parlett algorithm, a symmetric complete pivoting algorithm.  

\subsection{Finite precision analysis}
This section is devoted to the analysis of the RCP algorithm in finite precision. First we give some background on finite precision analysis. We follow the notation in \cite{higham2002accuracy}
\begin{equation*}
    fl(x \, op \, y) = (x \, op \, y)(1+\delta), \quad |\delta| \le u, \quad op = +,-,\times, /,
\end{equation*}
where $u$ is the unit roundoff error. We define constant
\begin{equation*}
    \gamma_n \stackrel{def}{=} \frac{n \, u}{1 - n \, u} = n \, u + O(u^2), \qquad \mbox{with} \quad n \, u < 1.
\end{equation*}

For matrix multiplication \cite{higham2002accuracy},
\begin{equation*}
    \left| fl(A \, B) - A \,B \right| \le \gamma_n |A| \, |B|, \quad \mbox{where} \quad A \in \mathbb{R}^{m \times n}, \quad B \in \mathbb{R}^{n\times q}.
\end{equation*}

Below is our theoretical result on finite precision analysis of random projection for RCP.
\begin{theorem} \label{Thm:finite precision analysis}
For Randomized complete pivoting method with SBKP strategy, and let $\tau(n, \epsilon)=12 \left( 1+\sqrt{\frac{1+\epsilon}{1-\epsilon}}\sqrt{n} \right)+1$, $0 < \epsilon < 1$, for $1\le k \le N-2$, we have
\begin{equation*}
    \left\|fl\left(B^{(k+1)}\right)-B^{(k+1)}\right\|_{1,2} 
 \le u \sqrt{p\left(1+\epsilon\right)} \left(\sqrt{n^3(1+\epsilon)} + 
 k \rho_{col} \tau(n, \epsilon)\right) \left\|A\right\|_{1,2} + O(u^2)
\end{equation*}
with probability at least $1-n(n+1) \exp{\left(-\frac{(\epsilon^2-\epsilon^3) p}{4}\right)} - \exp{\left(-\frac{\epsilon^2 \, p \, n}{2} \right)}$.
\end{theorem}
\begin{proof}
Without loss of generality, we will assume for this section only that the permutation matrix $\Pi = I$. In the case of a non-identity $\Pi$, simply apply the same analysis to $\Pi A \Pi^T$. Assume that we need to do $N$ steps to obtain the block $LDL^T$ factorization of $A$ and the block size used in each step is $s_i$ where $s_i \in \{1,2\}$ $(1 \le i \le N)$, the block $LDL^T$ factorization of $A$ can be written as 
\begin{align*}
A = \begin{pmatrix}
I &  & & \\
L_{21} & I & & \\
\vdots & \vdots  &\ddots & \\
L_{N1} & L_{N2} & \cdots & I 
\end{pmatrix} \begin{pmatrix}
D_1  &  &  &  \\
&  D_2 &  &  \\
&  &  \ddots &  \\
&   &   &   D_N
\end{pmatrix} \begin{pmatrix}
I & L_{21}^T  & \cdots &  L_{N1}^T \\
& I  &  \cdots &  L_{N2}^T  \\
&   &  \ddots  & \vdots \\
&   &   & I
\end{pmatrix} 
.
\end{align*}

Define $L_k \stackrel{def}{=} \begin{pmatrix}
L_{k+1,k}^T, L_{k+2,k}^T, \cdots, L_{N,k}^T
\end{pmatrix}^T$ ($1 \le k \le N-1$).
To begin, since $B^{(1)} = \Omega^{(1)} \, A^{(1)} = \Omega \, A$, then
$$\left|fl\left(B^{(1)}\right) - B^{(1)} \right| \le \gamma_n |\Omega| |A| = n \, u \, |\Omega| |A| + O(u^2).$$

For $k=1,2,\dots,N-2$, the update formula is $B^{(k+1)} = B_2^{(k)} - B_1^{(k)} L_k^T$, and $\left| \left(fl\left(B_1^{(k)}L_k^T\right)-B_1^{(k)}L_k^T\right) \right| \le 2 \, u \, \left|B_1^{(k)}\right| \left|L_k^T\right| + O(u^2)$, then
\begin{eqnarray*}
&&\left|fl\left(B^{(k+1)}\right)-B^{(k+1)}\right| \\
&\le& \left(1+u\right) \left|fl\left(B_2^{(k)}\right)-B_2^{(k)}\right| + u \, \left|B_2^{(k)}\right| + 3\, u \, \left|B_1^{(k)}\right| \left|L_k^T\right| + O(u^2).
\end{eqnarray*}
furthermore, for $1 \le k \le N-2$,
\begin{eqnarray*}
\left|fl\left(B^{(k+1)}\right)-B^{(k+1)}\right| &\le& \left(1+u\right) \left|fl\left(B^{(k)}\right)-B^{(k)}\right|\left(
\begin{array}{c}
     0  \\
     I 
\end{array}
\right) + 
u \left|B^{(k)}\right|\left(
\begin{array}{c}
     0  \\
     I 
\end{array}
\right) \\
&&+ 3\, u \, \left|B^{(k)}\right| \left(
\begin{array}{c}
     I  \\
     0 
\end{array}
\right)
\left|L_k^T\right| + O(u^2).
\end{eqnarray*}

Applying a simple induction argument to the inequality above gives
\begin{align*}
&\left|fl\left(B^{(k+1)}\right)  - B^{(k+1)}\right| 
\le \left| fl \left(B^{(1)} \right)  - B^{(1)}\right|\begin{pmatrix}
0\\I
\end{pmatrix} (1+u)^{k} \\
& + u  \sum_{i=1}^{k} \left| B^{(i)} \right| \begin{pmatrix}
0\\I
\end{pmatrix} (1+u)^{k-i} + 3 u \sum_{i = 1}^{k} \left| B^{(i)} \right| \begin{pmatrix}
I\\0
\end{pmatrix} \left| L_i^T \right| (1+u)^{k-i} + O(u^2),
\end{align*}
where $1 \le k \le N-2$.

We apply the $1,2$-norm on both sides to get
\begin{align*}
    &\left\| fl\left(B^{(k+1)}\right)  - B^{(k+1)} \right\|_{1,2} \le \left\| fl \left(B^{(1)} \right)  - B^{(1)} \right\|_{1,2} \, (1+u)^{k} \\
& + u \, \sum_{i=1}^{k} \left\| B^{(i)} \right\|_{1,2} \, (1+u)^{k-i} + 3 \, u \, \sum_{i = 1}^{k} \left\| B^{(i)} \right\|_{1,2} \left\| L_i^T \right\|_1 \, (1+u)^{k-i} + O(u^2).
\end{align*}

From equation \eqref{eq:condition between B and A}, we have 
\begin{equation*}
\left\|B^{(i)}\right\|_{1,2} \le \sqrt{p(1+\epsilon)} \left\|A^{(i)}\right\|_{1,2} \le \sqrt{p(1+\epsilon)}
\rho_{col}\left\|A\right\|_{1,2},
\end{equation*}
with probability of failure bounded above by $n(n+1) \exp{\left( -\frac{(\epsilon^2-\epsilon^3) p}{4} \right)}$, and by Theorem \ref{Thm:Upper Bounds on L entries},
\begin{equation*}
\left\|L_i^T\right\|_1 \le 2 \cdot \frac{1+\sqrt{\frac{1+\epsilon}{1-\epsilon}}\sqrt{n}}{{\bf min}\left(\alpha, 1-\alpha^2\right)} = 4 \left( 1+\sqrt{\frac{1+\epsilon}{1-\epsilon}}\sqrt{n} \right),
\end{equation*}
where the second quality holds because $\alpha = \frac{\sqrt{2}}{2}$ in Subsection \ref{Sec:Analysis of growth factor of RCP}, and
\begin{equation*}
\left\|fl\left(B^{(1)}\right)-B^{(1)}\right\|_{1,2} 
\le \gamma_n \left\|\left|\Omega\right|\right\|_2 \left\|A\right\|_{1,2} \le n \, u \, \left\|\Omega\right\|_F \left\|A\right\|_{1,2} + O(u^2).
\end{equation*}

Then 
\begin{eqnarray}\label{eq:1,2 norm of B(k+1) error}
&&\left\|fl\left(B^{(k+1)}\right)-B^{(k+1)}\right\|_{1,2} \notag \\
&\le&n \, u \, \left\|\Omega\right\|_F \left\|A\right\|_{1,2}  + 
k \, u \, \sqrt{p\left(1+\epsilon\right)} \rho_{col} \left\|A\right\|_{1,2} \notag \\
&&+ 12 \, k \, u \, \sqrt{p\left(1+\epsilon\right)} \rho_{col} \left\|A\right\|_{1,2} \left( 1+\sqrt{\frac{1+\epsilon}{1-\epsilon}}\sqrt{n} \right) + O(u^2) \notag \\
& \le & n \, u \, \left\|\Omega\right\|_F \left\|A\right\|_{1,2} + 
k \, u \, \sqrt{p\left(1+\epsilon\right)} \rho_{col} \left\|A\right\|_{1,2} \, \tau(n, \epsilon) + O(u^2).
\end{eqnarray}

We observe that function $h(\Omega) = \|\Omega \|_F$ is a Lipschitz function with Lipschitz constant $L = 1$ and apply Theorem 4.5.7 in \cite{bogachev1998gaussian} to get
\begin{equation*}
    \mathbb{P}\left\{ \|\Omega \|_F \ge \sqrt{pn} + t \right\} \le e^{-\frac{t^2}{2}},
\end{equation*}
where $\mathbb{E}\|\Omega\|_F \le \sqrt{\mathbb{E}\|\Omega\|_F^2} = \sqrt{pn}$ from Proposition 10.1 from \cite{halko2011finding}. Set $t = \epsilon \sqrt{pn}$, and plug this into \eqref{eq:1,2 norm of B(k+1) error} to arrive at our desired conclusion with an union bound.
\end{proof}

\subsection{Potential for numerical instability in random projections}

According to Theorem \ref{Thm:finite precision analysis}, updating random projections $B^{(k)}$ could potentially result in large rounding errors in $B^{(k)}$ in the theoretically possible case of large element growth. On the other hand, the input matrix $A$ could itself be nearly rank deficient, leading to potentially low-quality column pivots from $B^{(k)}$ even with small rounding errors in $B^{(k)}$. A similar numerical instability discussion on QR with column pivoting can be found in~\cite{drmavc2008failure}. 

We solve the problem with the large rounding error problem by explicitly recomputing a new random projection whenever necessary. We solve the rank deficiency problem by ensuring that pivoted column norms of $B$ remain above certain threshold value. The details are contained in Algorithm \ref{alg:Fixing}. 

With Algorithm \ref{alg:Fixing}, we need to recompute random projections at most $r$ times. In practice, we can simply set a particular $r$. In our implementation, we only set $r = 1$ (Algorithm \ref{alg:BRCP}) since the need for correcting numerical instability caused by updating random projections never arose in our experiments. 

\begin{algorithm}[H]
\caption{RCP with periodic random projection direct computation}
\label{alg:Fixing}
\begin{algorithmic}
\REQUIRE
\STATE symmetric matrix $A\in \mathbb{R}^{n \times n}$, sampling dimension $p~(p = 5)$, block size $b$,
\STATE $eps$ is machine precision, 
\STATE set a $r$ value.
\ENSURE
\STATE permutation matrix $\Pi$, unit lower triangular matrix $L$,
\STATE and block diagonal matrix $D$ such that $\Pi A \Pi^T = L D L^T$. \\
\algrule
\STATE generate random matrix $\Omega$ and compute $B = \Omega A$. Define $\beta = \left\|B\right\|_{1,2}$. Set $\delta = \frac{1}{r}$.
\WHILE{$\delta < 1$}
\STATE compute the current large column norm in $B$ and denote it as $t$. 
\IF{$t >= eps^\delta \, \beta$}
\STATE this pivot is acceptable. We apply this pivot to $A$, do a block $LDL^T$ factorization with SBKP strategy, update the $B$ matrix.
\ELSE
\STATE this pivot is unacceptable. We generate a new random matrix and compute a new random projection of the current Schur complement. Update $\delta = \delta + \frac{1}{r}$.
\IF{$t < eps^\delta \, \beta$}
\STATE the input matrix $A$ is rank deficient and the current Schur complement is close to zero matrix. Quit the while loop and return the computed $L$, $D$ and $\Pi$. 
\ELSE 
\STATE this pivot is acceptable. We apply this pivot to $A$, do a block $LDL^T$ factorization with SBKP strategy, update the $B$ matrix.
\ENDIF
\ENDIF
\STATE the matrix $A$ is either factorized already or the Schur complement of $A$ is close to zero matrix. 
\ENDWHILE
\end{algorithmic} 
\end{algorithm}

\section{Numerical Experiments}\label{Sec:Numerical experiment}
All experiments were run on a single 24-core node of the NERSC machine Edison. Subroutines were compiled using Cray Scientific Library. We compared RCP algorithm with LAPACK routines DSYSV, DSYSV\_ROOK and DSYSV\_AA for computing the solution to a real symmetric system of linear equations $Ax=b$. DSYSV uses Bunch-Kaufman (BK) algorithm, DSYSV\_ROOK uses the bounded Bunch-Kaufman (BBK) algorithm and DSYSV\_AA uses Aasen's algorithm.

We run experiments on ten different types of symmetric matrices $A\in \mathbb{R}^{n\times n}$. {\bf Type $4$} and {\bf Type $5$} matrices belong to the Matlab gallery. {\bf Type $8$} matrices are from the Higham's Matrix Computation Toolbox \cite{higham2002matrix}. {\bf Type $9$} matrices are from the University of Florida Sparse Matrix Collection \cite{davis2011university}. {\bf Type $10$} matrices are random rank-deficient matrices. All random matrices are generated using LAPACK subroutine DLARNV. For all tests, the right-hand side $b$ is chosen as $b=Ax$ with a random vector $x$. We compute relative backward error
\begin{equation} \label{Eq:Formula for BE}
\emph{err} \stackrel{def}{=} \frac{\| A \hat{x} - b \|_\infty}{\| A \|_\infty \| \hat{x} \|_\infty},
\end{equation}
where $\hat{x} $ is the computed solution. We also compute the growth factor  
\begin{equation} \label{Eq:Formula for GF of LDL}
\rho = \frac{\|D\|_{1,\infty}}{\|A\|_{1,\infty}},
\end{equation}
for all algorithms except Aasen's algorithm. For Aasen's algorithm, we compute the growth factor using an analogous definition
\begin{equation} \label{Eq:Formula for GF of LTL}
\rho = \frac{\|T\|_{1,\infty}}{\|A\|_{1,\infty}},
\end{equation}
where $T$ is the tridiagonal matrix computed by Aasen's algorithm.

\begin{itemize}
\item[{\bf Type $1$}:] $A$ is the worst-case matrix for element growth for Bunch-kaufman algorithm.
\[ A = \left(
\begin{array}{cc}
A_1 & (1 - \epsilon) I \\
(1 - \epsilon) I & O 
\end{array}
\right), \; \mbox{where} \; A_1 = \left(
\begin{array}{cccccc}
d_1 &  &  &  & 1 & 1 \\
& d_2 &  &  & 1 & 1 \\
&  & \ddots &  & \vdots & \vdots \\
&  &  & d_{n-2} & 1 & 1 \\
1 & 1 & \cdots & 1 & 1 & 1 \\
1 & 1 & \cdots & 1 & 1 & 1 
\end{array}
\right), \]
with ${\displaystyle 0 < \epsilon \ll 1, d_k = \frac{q^{1-k}}{1-q}(1 \le k \le n-2), q = 1 + \frac{1}{\alpha} \, \mbox{and} \, \alpha = \frac{1 + \sqrt{17}}{8}}$. For such matrix, Bunch-Kaufman algorithm results in exponential element growth \cite{druinsky2011growth}. 

\item[{\bf Type $2$}:] $A$ is the worst-case matrix for flops for the bounded Bunch-Kaufman algorithm.
$$
A = \left(
\begin{array}{ccccccc}
0 &  &  &  &  &  & 2 \\
& n & n &  &  &  &  \\
& n & 0 & n-1 &  &  &  \\
&  & n-1 & 0 &  &  &  \\
&  &  & \ddots & \ddots & \ddots &  \\
&  &  &  & 4 & 0 & 3 \\
2 &  &  &  &  & 3 & 0 
\end{array}
\right).
$$
For such matrix, the bounded Bunch-Kaufman algorithm results in a pivot search in the entire Schur complement in each step, leading to $O(n^3)$ extra work in comparison \cite{ashcraft1998accurate}.

\item[{\bf Type $3$}:] $A$ is a $n \times n$ Hankel matrix, i.e.,
$$
A = \begin{pmatrix}
a_1 & a_2 & a_3 & \cdots & \cdots & a_n \\
a_2 & a_3 &  &  &  & \vdots \\
a_3 &   &  &  &  & \vdots \\
\vdots &  &  &  &   & a_{2 n - 3} \\
\vdots &  &  &  & a_{2 n - 3} & a_{2 n - 2} \\
a_n & \cdots &  \cdots & a_{2 n - 3}   & a_{2 n - 2} & a_{2 n - 1},
\end{pmatrix},
$$
where $a_i \; (1 \le i \le 2 n - 1)$ are sampled independently from $\mathcal{N}(0, 1)$.

\item[{\bf Type $4$}:] $A$ is a discrete sine transform matrix of the form,
$$
A = \left(a_{ij}\right), \quad \mbox{where}  \quad a_{ij} = \sqrt{\frac{2}{n + 1}}  \sin \left( \frac{i j \pi}{n + 1} \right).
$$

\item[{\bf Type $5$}:] $A$ is a discrete cosine transform matrix of the form,
$$
A = \left(a_{ij}\right), \quad \mbox{where}  \quad a_{ij} = \cos \left( \frac{(i - 1) (j - 1) \pi}{n - 1} \right).
$$

\item[{\bf Type $6$}:] $A$ is a $n \times n$ Gaussian random matrix where the entries are sampled independently from $\mathcal{N}(0, 1)$.

\item[{\bf Type $7$}:] $A$ is a $n \times n$ KKT matrix \cite{nocedal1999springer},
$$
A = \left(
\begin{array}{cc}
A_1 & W \\
W^T & O 
\end{array}
\right),
$$
where $A_1 \in \mathbb{R}^{n_1\times n_1}$; $W \in \mathbb{R}^{n_1\times n_2}$ is a Gaussian random matrix where the entries are sampled independently from $\mathcal{N}(0, 1)$; $O \in \mathbb{R}^{n_2\times n_2}$ is a zero matrix, with $n = n_1 + n_2$.

\item[{\bf Type $8$}:] $A$ is a $n \times n$ augmented system matrix \cite{higham2002matrix},
$$
A = \left(
\begin{array}{cc}
I & W \\
W^T & O 
\end{array}
\right),
$$
where $I \in \mathbb{R}^{n_1\times n_1}$ is an identity matrix; $W \in \mathbb{R}^{n_1\times n_2}$ is a Gaussian random matrix where the entries are sampled independently from $\mathcal{N}(0, 1)$; $O \in \mathbb{R}^{n_2\times n_2}$ is a zero matrix, with $n = n_1 + n_2$.

\item[{\bf Type $9$}:] $A$ is from the University of Florida Sparse Matrix Collection \cite{davis2011university}. We chose $176$ real symmetric matrices, with sizes between $500$ and $10000$.

\item[{\bf Type $10$}:] $A$ is a $n \times n$ random rank-deficient matrix of the form,
$$
A = W \Lambda W^T,
$$
where $W \in \mathbb{R}^{n\times n}$ is a Gaussian random matrix where the entries are sampled independently from $\mathcal{N}(0, 1)$; $\Lambda = \mbox{diag}(\lambda_1,\cdots, \lambda_n)$ is a diagonal matrix with $\lambda_i = \frac{q^{1-i}}{1-q},~q=1+\sqrt{2}~(1\le i \le n)$. 
\end{itemize}

Figure \ref{Fig:Relative run time between RCP and BK} shows the relative runtime ratios of RCP algorithm and Bunch-Kaufman algorithm (Aasen's algorithm) for {\bf Type $6$} matrices. As matrix size $n$ increases, the relative run times decrease to a negligible amount. Figure \ref{Fig:Four factors for random matrices with different p} shows the values of $p$ used in RCP algorithm have little impact on growth factor, backward error, $\|L\|_1$ and $\|L^{-1}\|_1$ for {\bf Type $6$} matrices. We get the same results for other types of matrices, in terms of relative runtime ratios and impact of $p$. We choose $p=5$ in Figure \ref{Fig:worst case growth factor for BK} - \ref{Fig:Results for rank deficient matrices}.

Figure \ref{Fig:worst case growth factor for BK} shows the growth factor of the $L$ computed by BK algorithm increases exponentially and therefore the algorithm breaks down. RCP algorithm is as efficient as bounded Bunch-Kaufman algorithm and Aasen's algorithm and meanwhile very stable. In Figure \ref{Fig:Running time of the worst case for BBK}, the right figure contains all curves except bounded Bunch-Kaufman curve of the left figure, to show difference between those three methods. As the matrix size increases, the run time of bounded Bunch-Kaufman algorithm increases exponentially, while the run times of all the other three algorithms remain small.

Figure \ref{Fig:results of element growth factor} shows the growth factor and backward error for matrices of {\bf Type $3$} through {\bf $8$}. RCP algorithm is better than all the other algorithms on both growth factor and backward error. RCP algorithm is also better than BK algorithm in terms of $L$ and $L^{-1}$ factors in Figure
\ref{Fig:results of L}.

For {\bf Type $9$} matrices, backward error, growth factor, $\|L\|_1$ and $\|L^{-1}\|_1$ are shown in Figure \ref{Fig:results of UF sparse matrix}. The results computed by RCP algorithm are comparable to the other three algorithms, while RCP algorithm produces $L~(L^{-1})$ matrices with much smaller $\|L\|_1~(\|L^{-1}\|_1)$.

Figure \ref{Fig:Results for rank deficient matrices} compares the growth factor, backward error, $\|L\|_1$ and $\|L^{-1}\|_1$ for {\bf Type $10$} matrices, whose ranks are less than $55~(n \ge 100)$ and have stability solutions of the linear systems because of rounding error. The growth factor are all $1$ in these four algorithms. The backward error of RCP algorithm is comparable to other three algorithms, while $\|L\|_1~(\|L^{-1}\|_1)$ are as small as BBK algorithm and Aasen's algorithm, and much smaler than BK algorithm.

\begin{figure}[htbp]
\centering
\includegraphics[width=1.0\textwidth]{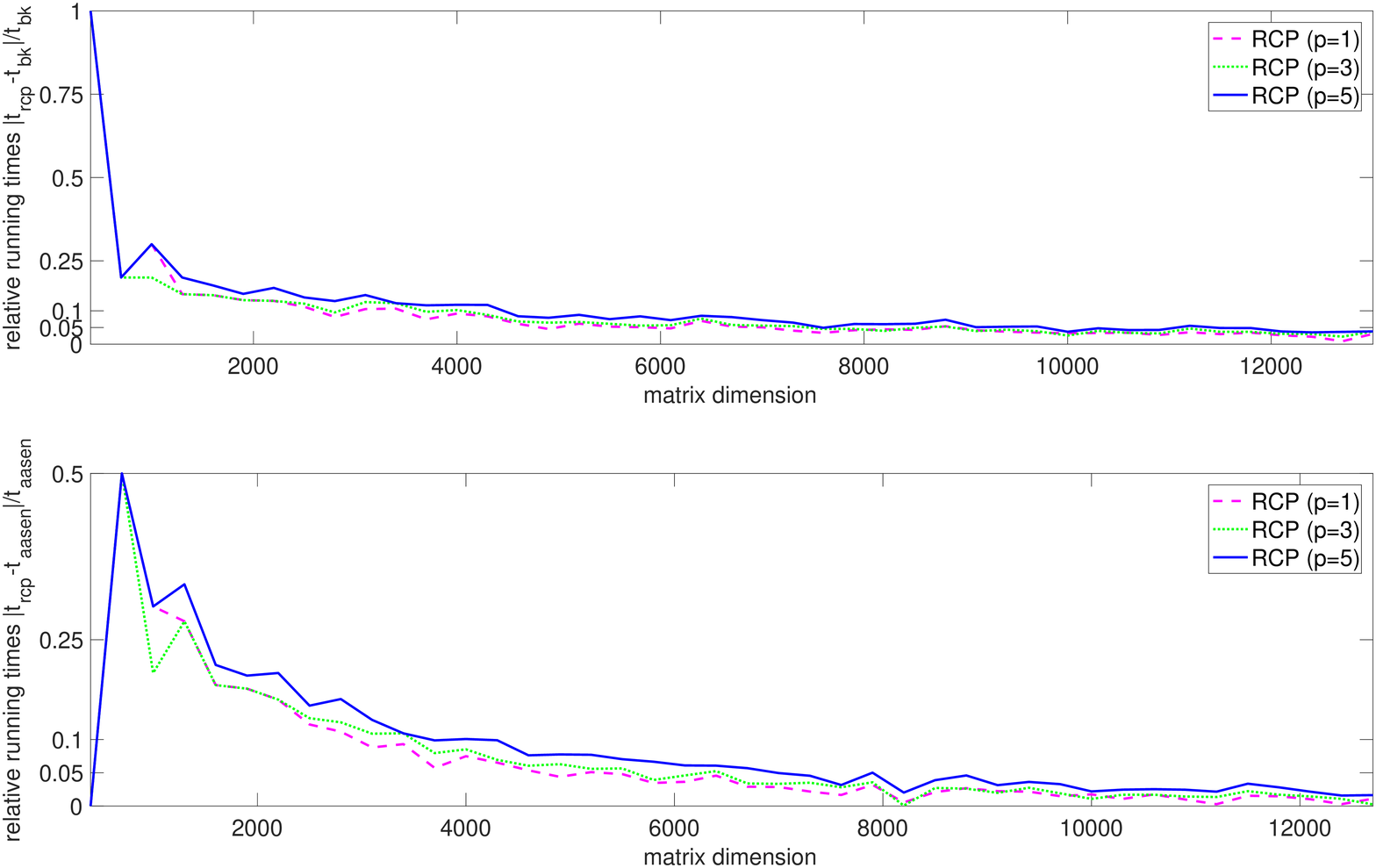}
\caption{Relative running times between RCP with three different value of $p$ and BK (Aasen)  for {\bf Type $6$} matrices.}\label{Fig:Relative run time between RCP and BK}
\end{figure}

\begin{figure}[htbp]
\centering
\includegraphics[width=1.0\textwidth]{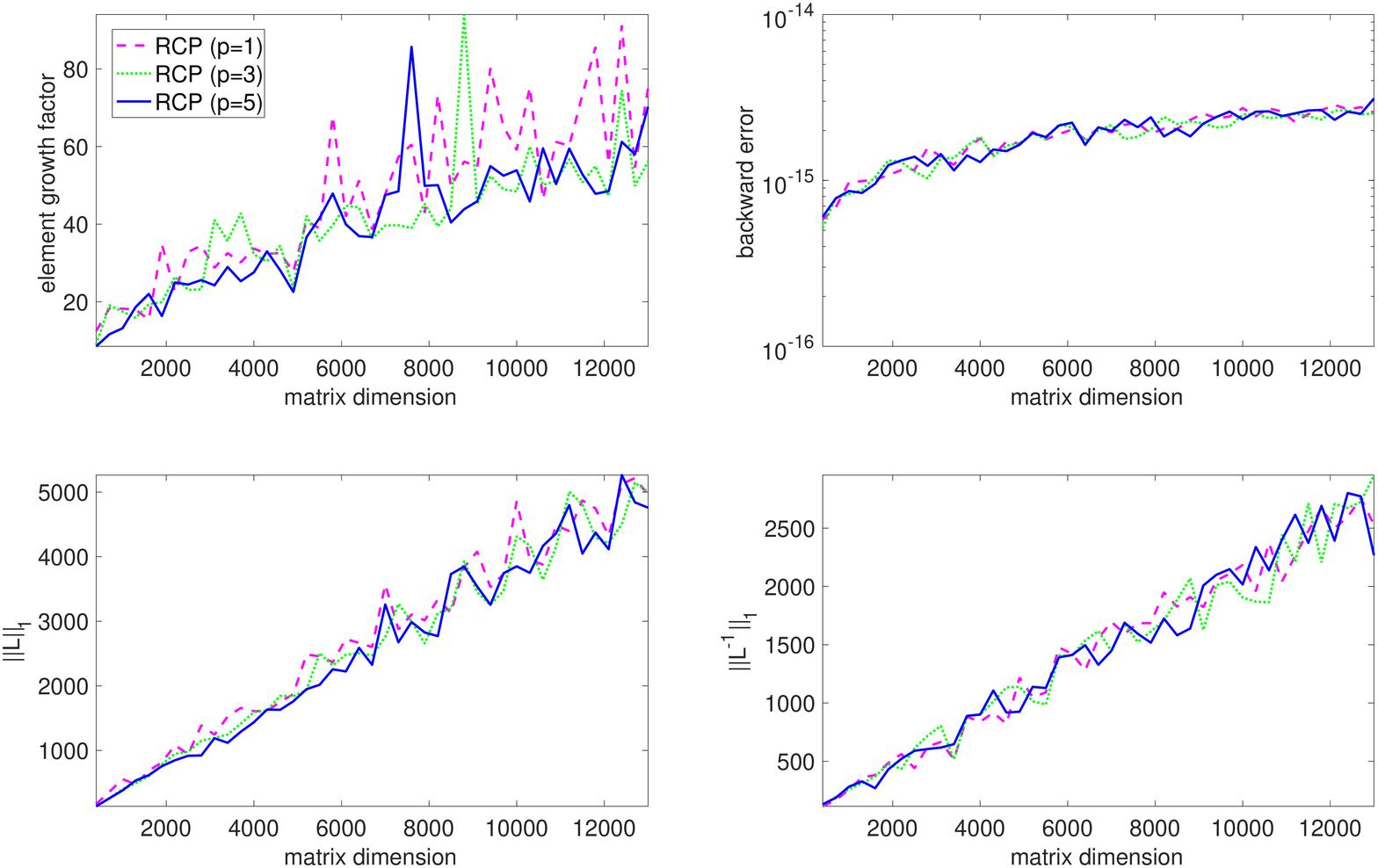}
\caption{Results for {\bf Type $6$} matrices with different $p$. The top left plot shows the growth factor $\rho$ in the factorization of $A$ as defined by \eqref{Eq:Formula for GF of LDL}, the top right one shows the backward error in the solution of $Ax=b$ as defined by \eqref{Eq:Formula for BE}, the bottom left one shows the $1$-norm of $L$, and the bottom right one shows the $1$-norm of inverse of $L$.}\label{Fig:Four factors for random matrices with different p}
\end{figure}

\begin{figure}[htbp]
\centering
\includegraphics[width=1.0\textwidth]{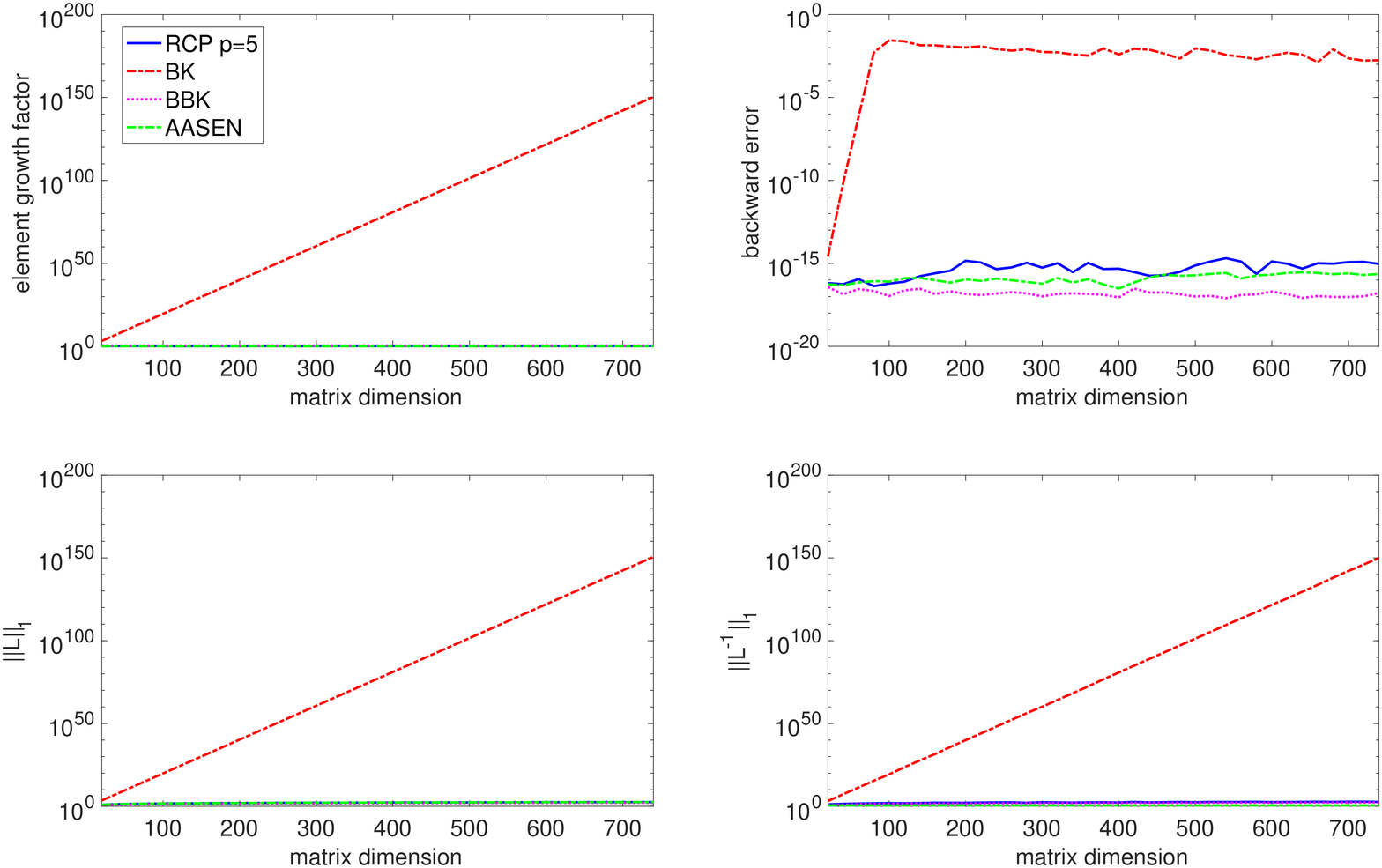}
\caption{Results for {\bf Type $1$} matrices. The top left plot shows the growth factor $\rho$ in the factorization of $A$ as defined by \eqref{Eq:Formula for GF of LDL} and \eqref{Eq:Formula for GF of LTL}, the top right one shows the backward error in the solution of $Ax=b$ as defined by \eqref{Eq:Formula for BE}, the bottom left one shows the $1$-norm of $L$, and the bottom right one shows the $1$-norm of inverse of $L$.}\label{Fig:worst case growth factor for BK}
\end{figure}

\begin{figure}[htbp]
\centering
\includegraphics[width=1.0\textwidth]{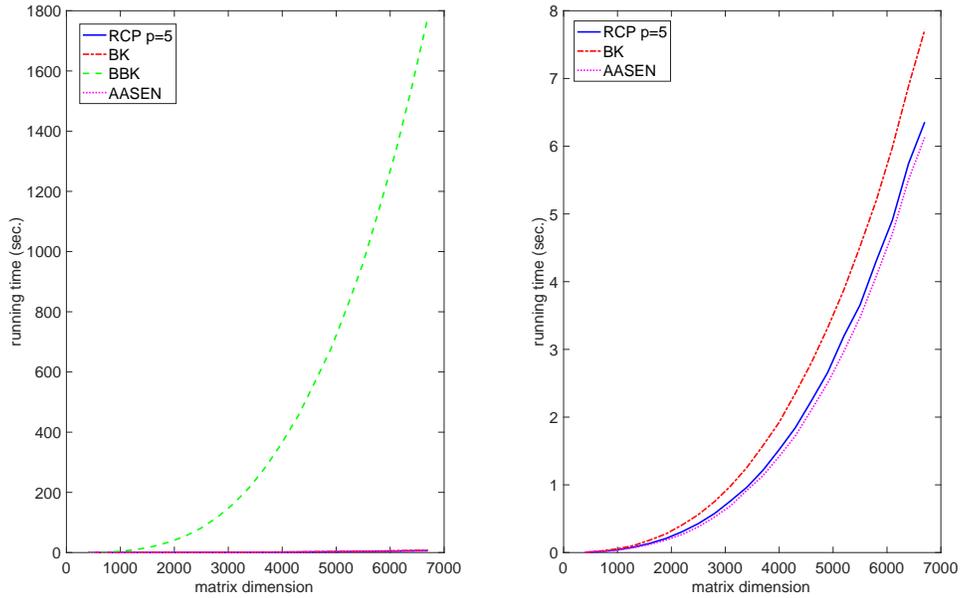}
\caption{Running times for {\bf Type  $2$} matrices.}\label{Fig:Running time of the worst case for BBK}
\end{figure}

\begin{figure}[htbp]
\centering
\includegraphics[width=1.0\textwidth]{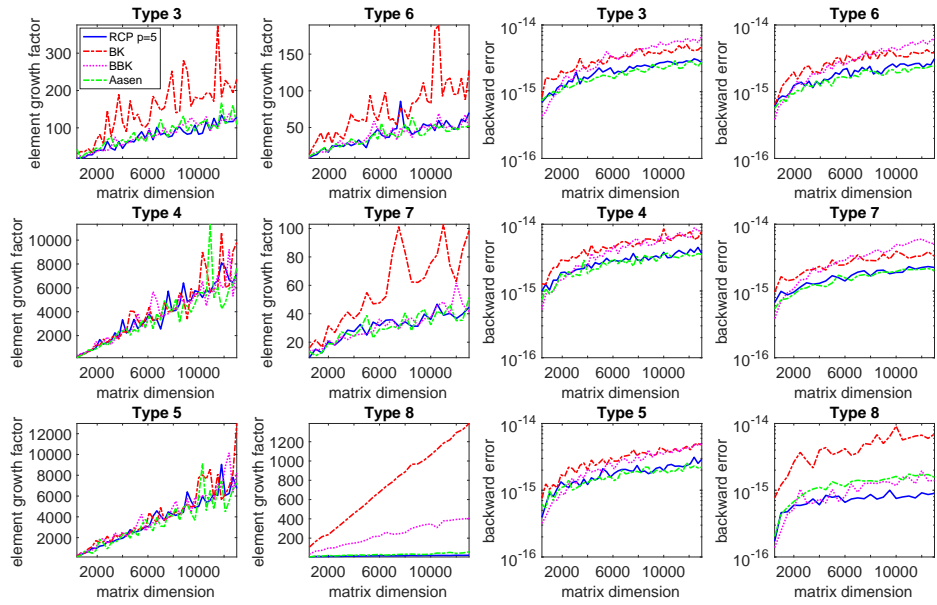}
\caption{Growth factor and backward error for matrices of {\bf Type $3$} through {\bf $8$}. The left half plots show the element growth factor in the factorization of $A$ as defined by \eqref{Eq:Formula for GF of LDL} and \eqref{Eq:Formula for GF of LTL}, and the right half ones show the backward error as defined by \eqref{Eq:Formula for BE}. Element growth factors of {\bf Type $4$}, {\bf Type $5$} and {\bf Type $8$} increase linearly, and the other three are small.}\label{Fig:results of element growth factor}
\end{figure}

\begin{figure}[htbp]
\centering
\includegraphics[width=1.0\textwidth]{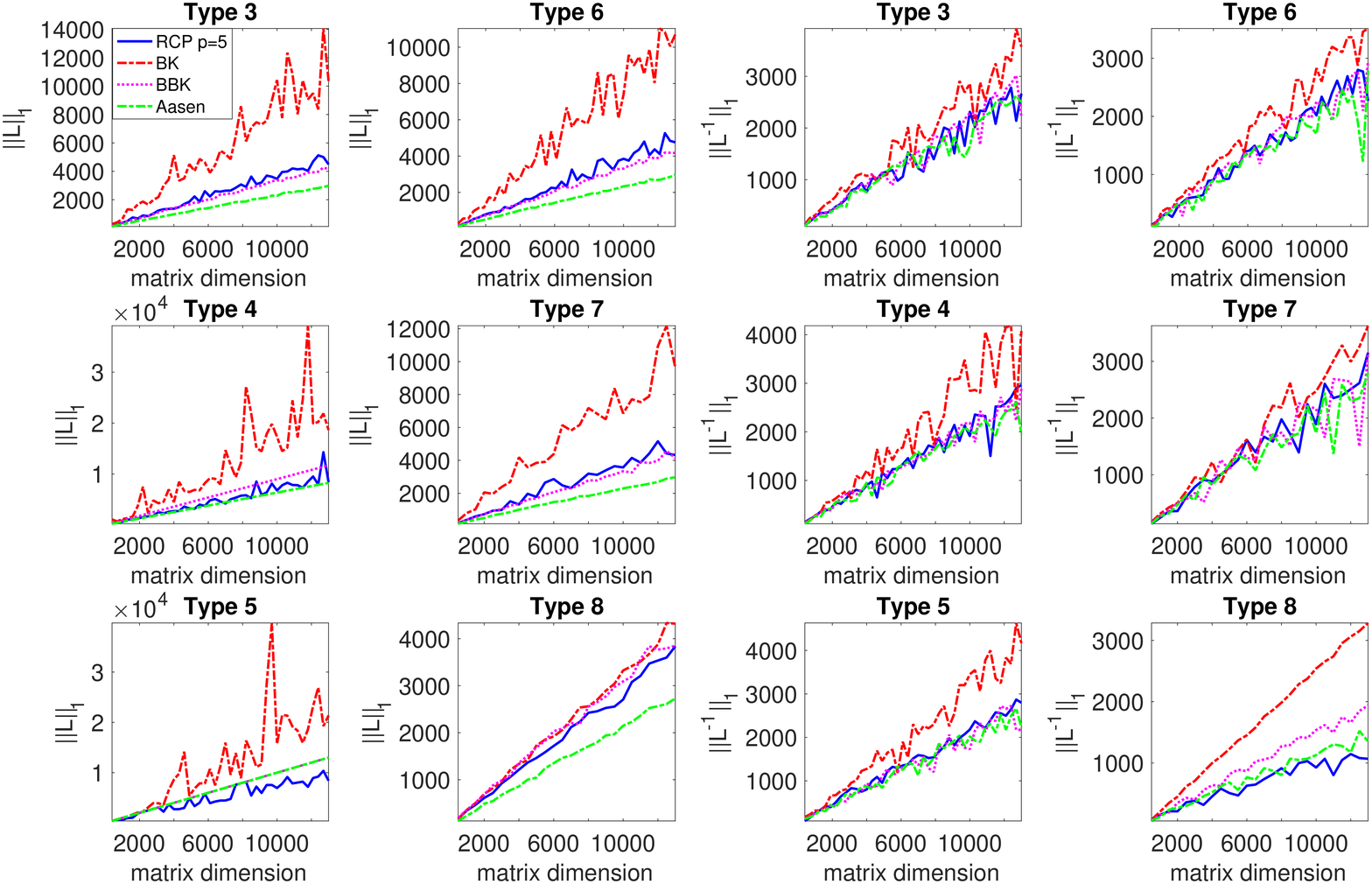}
\caption{$\|L\|_1$ and $\|L^{-1}\|_1$ for matrices of {\bf Type $3$} through {\bf $8$}. The left half plots show the values of $\|L\|_1$, and the right half ones show the value of $\|L^{-1}\|_1$. All these results increase linearly.}\label{Fig:results of L}
\end{figure}

\begin{figure}[htbp]
  \centering
  \label{Fig:results of UF sparse matrix}\includegraphics[width=1.0\textwidth]{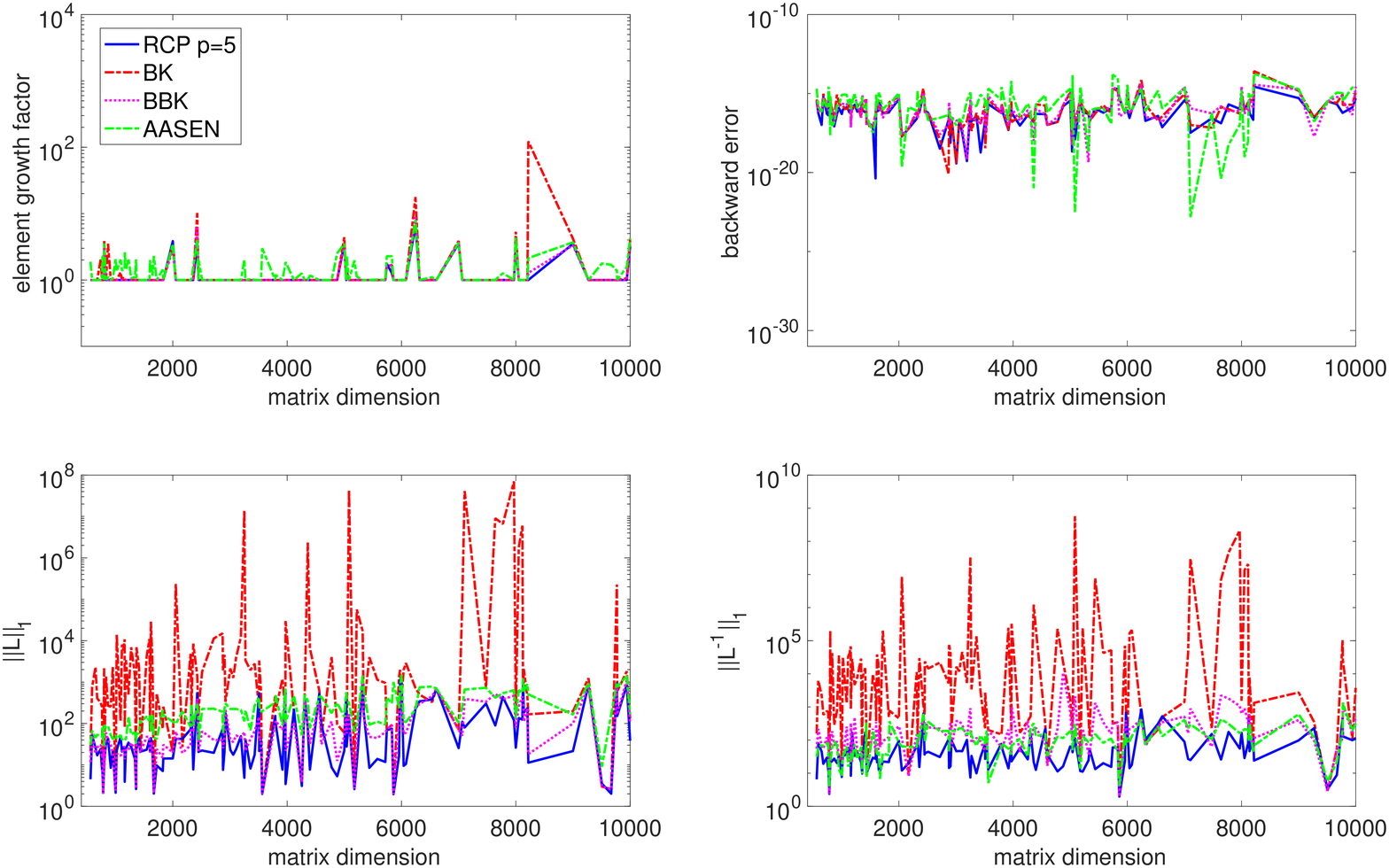}
  \caption{Results for {\bf Type $9$} matrices. The top left plot shows the growth factor $\rho$ in the factorization of $A$ as defined by \eqref{Eq:Formula for GF of LDL} and \eqref{Eq:Formula for GF of LTL}, the top right one shows the backward error in the solution of $Ax=b$ as defined by \eqref{Eq:Formula for BE}, the bottom left one shows the $1$-norm of $L$, and the bottom right one shows the $1$-norm of inverse of $L$.}
\end{figure}

\begin{figure}[htbp]
\centering
\includegraphics[width=1.0\textwidth]{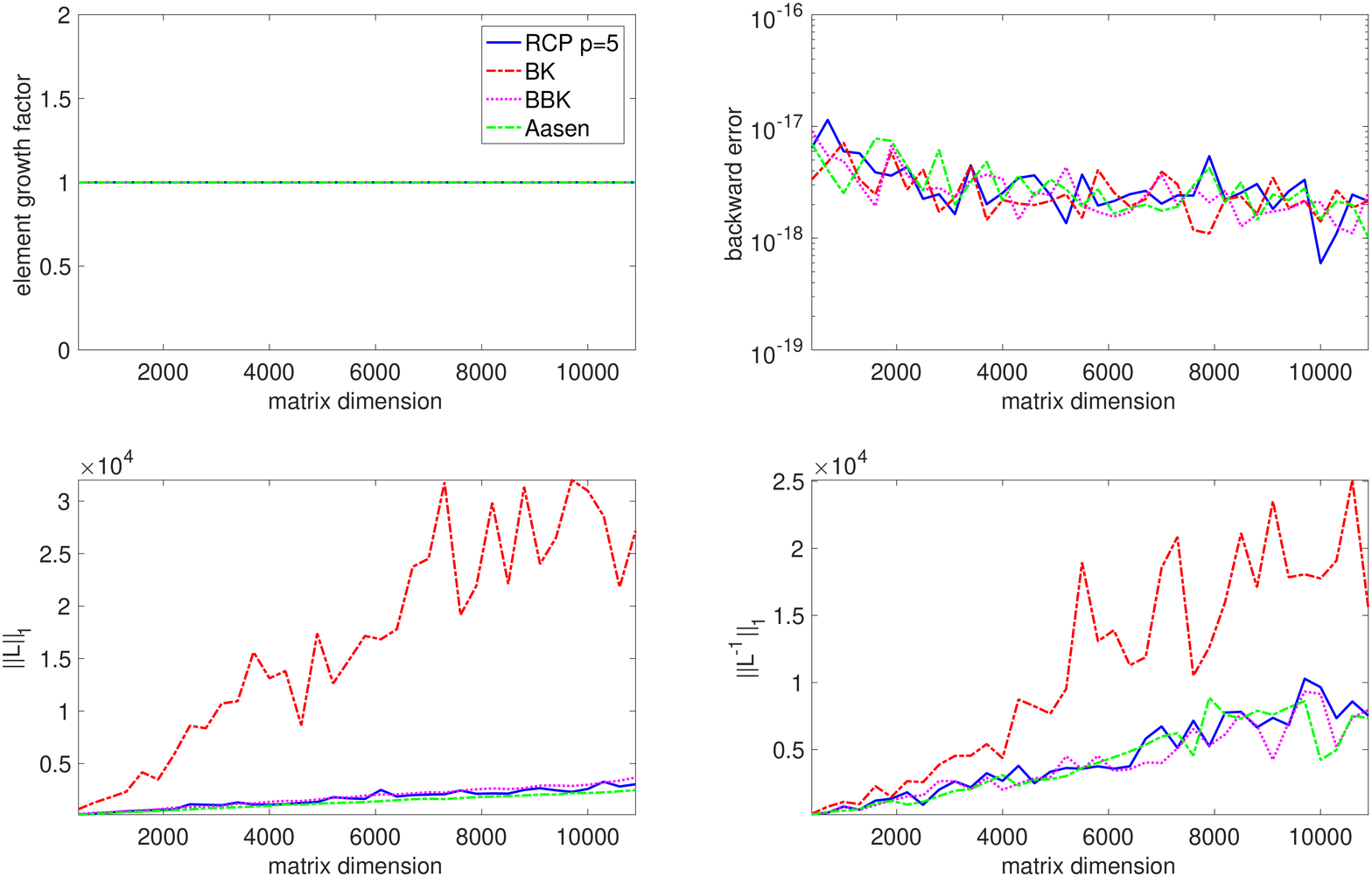}
\caption{Results for {\bf Type $10$} matrices. The top left plot shows the growth factor $\rho$ in the factorization of $A$ as defined by \eqref{Eq:Formula for GF of LDL} and \eqref{Eq:Formula for GF of LTL}, the top right one shows the backward error in the solution of $Ax=b$ as defined by \eqref{Eq:Formula for BE}, the bottom left one shows the $1$-norm of $L$, and the bottom right one shows the $1$-norm of inverse of $L$.}\label{Fig:Results for rank deficient matrices}
\end{figure}

\section{Conclusions}\label{Sec:Conclusions}
In this paper, we have introduced the randomized complete pivoting (RCP) algorithm for computing the $LDL^T$ factorization of a symmetric indefinite matrix. We developed a theoretical high-probability element growth factor upper bound for RCP, that is compatible to that of GECP. Moreover, we performed an error analysis to demonstrate the numerical stability of RCP. In our numerical experiments, RCP is as stable as Gaussian elimination with complete pivoting, yet only slightly slower than Bunch-Kaufman algorithm and Aasen's algorithm.

\bibliographystyle{siamplain}
\bibliography{references}
\end{document}